\documentclass[12pt, reqno]{amsart}
\usepackage{amssymb,a4}
\usepackage{amsmath,amsfonts,amssymb,amsthm}
\usepackage{mathrsfs}
\usepackage{color}
\usepackage{colordvi}

\newcommand{\on}{\operatorname}
\newcommand{\cal}{\mathcal}
\newcommand{\f}{\mathfrak}
\newcommand{\mbf}{\mathbf}

\newcommand{\al}{\alpha}

\def\wt{{\rm wt}}

\def\de{\delta}
\def\be{\beta}

\newcommand{\la}{\lambda}

\def\C{{\mathbb C}}
\def\Q{{\mathbb Q}}

\def\Z{{\mathbb Z}}

\def\N{{\mathbb N}}
\def\1{{\bf 1}}
\def\l{\lambda}

\def \End{{\rm End}}

\def \Hom{{\rm Hom}}
\def \Irr{\on{Irr}}

\def \<{\langle}
\def \>{\rangle}
\def \w{\omega}
\def \wg{{\widehat{\frak{g}}}}

\def \sl{\frak{sl}}
\def \a{\alpha}
\def \b{\beta}
\def \e{\epsilon}
\def \h{\mathfrak{h}}
\def \l{\lambda}
\def \w{\omega}
\numberwithin{equation}{section}

\newtheorem{theorem}{Theorem}[section]

\newtheorem{lem}[theorem]{Lemma}
\newtheorem{cor}[theorem]{Corollary}
\newtheorem{remark}[theorem]{Remark}
\theoremstyle{definition}
\newtheorem{defn}[theorem]{Definition}

\begin{document}
\title[Vertex Operator Algebras]{Tensor Decomposition, Parafermions,  Level-Rank Duality, and Reciprocity Law for Vertex Operator Algebras}

\author{ Cuipo (Cuibo) Jiang} \thanks{Jiang is supported by China NSF grants 10931006, 11371245,  China RFDP grant 2010007310052, and the Innovation Program of Shanghai Municipal Education Commission (11ZZ18)}
\address[Jiang]{Department of Mathematics,  Shanghai Jiaotong University, Shanghai, 200240, China }
\email{cpjiang@sjtu.edu.cn}

\author{Zongzhu Lin}
\address[Lin]{ Department of Mathematics, Kansas State University, Manhattan, KS 66506, USA}
\email{zlin@math.ksu.edu}

\begin{abstract}  For the semisimple Lie algebra $ \frak{sl}_n$, the basic
 representation $L_{\widehat{\frak{sl}_{n}}}(1,0)$ of the affine Lie algebra
 $\widehat{\frak{sl}_{n}}$ is a lattice vertex operator algebra.
 The first main result of the paper is to prove that the commutant vertex operator
  algebra of $ L_{\widehat{\frak{sl}_{n}}}(l,0)$ in the $l$-fold tensor product
  $ L_{\widehat{\frak{sl}_{n}}}(1,0)^{\otimes l}$ is isomorphic to the parafermion
  vertex operator algebra $K(\frak{sl}_{l},n)$, which is the commutant of the
Heisenberg vertex operator algebra $L_{\widehat{\frak{h}}}(n,0) $ in $L_{\widehat{\frak{sl}_l}}(n,0)$.
The result provides a version of level-rank duality. The second main result of the
paper is to prove more general version of the first result that the commutant of  $ L_{\widehat{\frak{sl}_{n}}}(l_1+\cdots +l_s, 0)$ in
$L_{\widehat{\frak{sl}_{n}}}(l_1,0)\otimes \cdots \otimes L_{\widehat{\frak{sl}_{n}}}(l_s, 0)$ is isomorphic to the commutant of the vertex operator algebra generated by a
Levi Lie subalgebra of $\frak{sl}_{l_1+\cdots+l_s}$ corresponding to the
composition
$(l_1, \cdots, l_s)$ in the rational vertex operator algebra
$ L_{\widehat{\frak{sl}}_{l_1+\cdots +l_s}}(n,0)$. This general version
also resembles a version of reciprocity law discussed by Howe in the context
of reductive Lie groups. In the course of the proof of the main results, certain
Howe duality pairs also appear in the context of vertex operator algebras.

\end{abstract}

\maketitle
\section{Introduction}

 \subsection{} Given a finite dimensional simple Lie algebra $\frak{g}$,
  let $\wg=\frak{g}\otimes [t,t^{-1}]\oplus  \C K$ be its affine Lie algebra.
  Let  $ L_{\wg}(1,0)$ be the basic representation of $\wg$. Then for any
  $l \in \N$, the tensor product $\wg$-module $L_\wg(1,0) ^{\otimes l} $ is a direct sum of irreducible $\wg$-modules and there is a decomposition
 \begin{equation} \label{eq:decomposition} L_{\wg}(1,0)^{\otimes l}=\bigoplus L_\wg(l, \bar{\Lambda})\otimes_{\C}M_\wg(l, \bar{\Lambda})
 \end{equation}
where $L_{\wg}(l, \bar{\Lambda})$ are level $l$ irreducible  $\wg$-modules and $M_\wg(l, \bar{\Lambda})=\Hom_{\wg}(L_\wg(l, \bar{\Lambda}), L_{\wg}(1,0)^{\otimes l})$ are vector spaces.
  Determining
 $M_\wg(l, \bar{\Lambda})$ is one of
the main problems in representation theory which is amount to decomposing
the tensor product. Unlike the question for finite dimensional Lie algebras
(or corresponding Lie groups, quantum groups, etc), the vector spaces
$M_\wg(l, \bar{\Lambda})$ are infinite dimensional.  The obvious highest weight vector
$v^+\otimes \cdots \otimes v^{+}$ in $L_{\wg}(1,0)^{\otimes l}$ generates
an irreducible $\wg$-submodule isomorphic to  $L_{\wg}(l,0)$.
Thus $ M_\wg(l,0)\neq 0$. Since $L_{\wg}(1,0)$ and $L_{\wg}(l,0)$ are
vertex operator algebras, $L_{\wg}(1,0)^{\otimes l}$ has a tensor product
vertex operator algebra structure with $L_{\wg}(l,0)$ being a vertex operator subalgebra
(with  a different conformal vector).  In general, for a vertex operator algebra  $V$ and a vertex operator subalgebra $U$ (with possibly different conformal vectors) of $V$, we denote by $ C_V(U)$ the commutant of $ U$ in $V$ (see \cite[3.11]{LL} and $C_V(U)$  is a vertex operator subalgebra of $V$ (with possibly different conformal vectors). Then
$M_\wg(l,0)=C_{L_{\wg}(1,0)^{\otimes l}}(L_{\wg}(l,0))$ is the commutant
subalgebra of $L_{\wg}(l,0)$ in $L_{\wg}(1,0)^{\otimes l}$  and  is  a simple
vertex operator subalgebra of $L_{\wg}(1,0)^{\otimes l} $ and
$M_\wg(l,\bar{\Lambda})$ are
$M_\wg(l,0)$-modules.  In this paper,
 we first prove
 \begin{theorem} \label{thm:1.1}
   
 Let  $L_{\widehat{\frak{h}}}(n,0)$ be the vertex operator subalgebra of $L_{\widehat{\frak{sl}_{l}}}(n,0)$ generated by  the Cartan subalgebra $\frak{h}$ of $\frak{sl}_l$. Then $C_{L_{\widehat{\frak{sl}_n}}(1,0)^{\otimes l}}(L_{\widehat{\frak{sl}_n}}(l,0)) \cong
 C_{L_{\widehat{\frak{sl}_l}}(n,0)}(L_{\widehat{\frak{h}}}(n,0))$ as vertex operator algebras.
 \end{theorem}

Note that $L_{\widehat{\frak{sl}_{l}}}(n,0)$
is an $\widehat{\frak{sl}_{l}}$-module of level $ n$ while
$L_{\widehat{\frak{sl}_n}}(l,0)$ is an $ \widehat{\frak{sl}_n}$-module of level $l$.
This provides another level-rank duality.  The commutant $C_{L_{\widehat{\frak{sl}_l}}(n,0)}(L_{\widehat{\frak{h}}}(n,0))$ is the parafermion vertex operator algebra by physcists and denoted by $ K(\f{sl}_l, n)$ which have been extensively studied \cite{ZF}, \cite{CGT}, \cite{DLY2}, \cite{DLWY}, \cite{ALY}, \cite{DW1}-\cite{DW3}. However, $C_{L_{\widehat{\frak{sl}_l}}(n,0)}(K(\f{sl}_l, n))$ only contains $L_{\widehat{\frak{h}}}(n,0)$ and is  a conformal extension of $ L_{\widehat{\frak{h}}}(n,0)$ in $ L_{\widehat{\frak{sl}_l}}(n,0)$. The proof of Theorem~\ref{thm:1.1}  also shows that 
\begin{equation} C_{L_{\widehat{\frak{sl}_l}}(n,0)}(K(\f{sl}_l, n))\cong V_{\sqrt{n}A_{l-1}}
\end{equation}   which is a lattice vertex operator algebra corresponding to the  lattice  $ \sqrt{n}A_{l-1}$. This follows from  \eqref{decomp1}. Therefore $(K(\frak{sl}_l, n), V_{\sqrt{n}A_{l-1}})$ is a duality pair in the vertex operator algebra $L_{\widehat{\frak{sl}_{l}}}(n,0) $ in the following sense of \cite{howe:reciprocity}.

\subsection{} More generally, given a sequence of positive integers
$\underline{\ell}=(l_1, \cdots l_s)$, the tensor product vertex operator algebra
$L_{\wg}(\underline{\ell}, 0)= L_{\wg}(l_1, 0)\otimes L_{\wg}(l_2, 0)\otimes
\cdots \otimes L_{\wg}(l_s,0)$
contains a vertex operator subalgebra isomorphic to
$ L_\wg(|\underline{\ell}|,0)$ with $|\underline{\ell}|=l_1+\cdots+l_s$. One considers the $ \wg$-module decomposition:
\begin{equation}   L_{\wg}(\underline{\ell}, 0)=\bigoplus L_{\wg}(|\underline{\ell}|, \bar{\Lambda})
\otimes M_\wg(\underline{\ell}, \bar{\Lambda}).
\end{equation}
Then $ M_\wg(\underline{\ell}, 0)=C_{L_{\wg}(\underline{\ell}, 0)}(L_{\wg}(|\underline{\ell}|, 0))$
is a simple vertex operator algebra and all  $M_\wg(\underline{\ell}, \bar{\Lambda})$ are
$ M_\wg(\underline{\ell}, 0)$-modules. On the other hand, the sequence
$\underline{\ell}$ defines a Levi subalgebra $\frak{l}_{\underline{\ell}}$ of
$\frak{sl}_{|\underline{\ell}|}$. The vertex operator subalgebra of
$ L_{\widehat{\frak{sl}_{|\underline{\ell}|}}}(n, 0)$ generated by
$\frak{l}_{\underline{\ell}}$ is simple (but not rational) and is
denoted by $L_{\widehat{\frak{l}_{\underline{\ell}}}}(n,0)$.
Set $K(\frak{sl}_{|\underline{\ell}|}, \frak{l}_{\underline{\ell}}, n)=
C_{L_{\widehat{\frak{sl}_{|\underline{\ell}|}}}(n, 0)}(L_{\widehat{\frak{l}_{\underline{\ell}}}}(n, 0)).$
This commutant subalgebra was first mentioned in \cite{FZ}. Note that 
$ \frak{l}_{\underline{\ell}}=\frak{h}$ when $\underline{\ell}=(1, \cdots, 1)$ and $K(\frak{sl}_{|\underline{\ell}|}, \frak{l}_{\underline{\ell}}, n)=K(\frak{sl}_{|\underline{\ell}|},  n)$ is the parafermion.   With this setting, a more general rank-level duality holds.

\begin{theorem} \label{thm:1.2} If $ \frak{g}=\frak{sl}_n$, then $C_{L_{\widehat{\frak{sl}_n}}(\underline{\ell}, 0)}(L_{\widehat{\frak{sl}_n}}(|\underline{\ell}|, 0))\cong
C_{L_{\widehat{\frak{sl}_{|\underline{\ell}|}}(n,0)}}(L_{ \widehat{\frak{l}_{\underline{\ell}}}}(n,0))$ as vertex operator algebras.
\end{theorem}

Similarly, the proof (cf. \eqref{levi-decomp1} ) of Theorem~\ref{thm:1.2} also shows that 
$C_{L_{\widehat{\frak{sl}_{|\underline{\ell}|}}}(n, 0)}(K(\frak{sl}_{|\underline{\ell}|}, \frak{l}_{\underline{\ell}}, n)) 
$  is isomorphic to the tensor product of the vertex operator algebra corresponding to a semi-simple Lie algebra and a lattice vertex operator algebra, and in particular is a rational vertex operator algebra. 
\subsection{} Decomposing tensor products of
representations has a long history in various representation
 theories such as   finite groups, Lie groups, algebraic groups, quantum groups  (and super-version) and many others to come.
 It is related to invariant theory, combinatorics, and many other fields.
 For example, the classical problem of decomposing the tensor
 power of the natural representation of classical Lie algebras
 (Lie groups, algebraic groups) has been in the heart of
 representation theory for over a century and it is still the driving motivation today for their quantum version and super-version. Although no single reference would do the justice for subject, we do refer to Weyl's book \cite{Weyl} and Howe's lecture \cite{H},  for much earlier work and its influence in mathematics, and to \cite{LZ} for most recent results.  In those cases, the component corresponding to one-dimensional trivial module is of the main
 interests and is the invariant subspace. However, for the affine
 Lie algebra $ \widehat{\f{g}}$, the trivial module does not appear in
 the non-critical level, and  the module $L_{\widehat{\f{g}}}(l,0)$ plays the role
 of  the identity object in the fusion tensor category and thus, the
 commutant $ C_{L_{\wg}(1,0)^{\otimes l}}(L_{\wg}(l,0))$ can be
 regarded as the invariant, which is a vertex operator algebra.
Describing the commutant algebras in the algebras of operators on tensor
powers  is a main goal in invariant theory (see \cite{LZ} for recent results).
They are all related to Hecke algebras, Brauer algebras etc. In the classical
cases, the commutant algebra is the invariant subalgebra of the
endomorphism algebra of the tensor vector space.

In terms of tensor
category $\mathcal{C}$ of representations, the commutant algebra is
\begin{equation} \label{centralizer}\on{End}_{\mathcal{C}}(V^{\otimes l})=
\Hom_{\mathcal{C}}(\mathbb{I}, (V^{\otimes l})^{\vee}\otimes V^{\otimes l})\end{equation}
with $\mathbb{I}$ being the identity object of the tensor category
$ \mathcal{C}$. However, in case of  vertex
operator algebras, the commutant is
\begin{equation} \label{commutant}C_{L_{\wg}(1, 0)^{\otimes l}}(L_{\wg}(l, 0))=\Hom_{L_{\wg}(l, 0)}(L_{\wg}(l, 0), L_{\wg}(1, 0)^{\otimes l}),
\end{equation}
where $ L_{\wg}(1, 0)^{\otimes l})$ is a vertex operator algebra
while the tensor product is taken in the tensor category of
representations of the affine Lie algebra $\wg$, rather than in
the fusion tensor category of the vertex operator algebra.
However $L_{\wg}(l, 0)$ is the identity object of the fusion
tensor category of the vertex operator algebra $L_{\wg}(l, 0)$ and all other coefficients
\begin{equation} \label{schur functor} M_{\wg}(l, \bar{\Lambda})=\Hom_{L_{\wg}(l,0)}(L_{\wg}(l,\bar{\Lambda}), L_{\wg}(1,0)^{\otimes l})
\end{equation}
are  $ C_{L_{\wg}(1, 0)^{\otimes l}}(L_{\wg}(l, 0))$-modules. From this aspect, the commutant vertex operator algebra  $C_{L_{\wg}(1, 0)^{\otimes l}}(L_{\wg}(l, 0))$ still plays the role of the endomorphism algebras as in the classical cases mentioned above. Thus describing the commutant algebras in tensor powers seems to be an invariant theory in the context of vertex operator algebras.

It is well known that commutant construction of vertex operator algebra is the so-called coset construction in conformal field theory. The main question of computing the commutant subalgebras ware first raised in \cite{FZ}.

\subsection{} The classical Schur-Weyl duality builds the relation
of the decomposition of the tensor power $V^{\otimes l}$ with
the representations of the endomorphism ring. This relation builds
connection between quite different representation theories, such as
representations of general linear groups (algebraic, quantum,
or their super version) with combinatorics (and geometry). Through this
relations many properties of the representation theory of one side can be
 described by the other side. See \cite{CLW} for recent application of this
 duality in characterization of the irreducible characters for certain Lie super
 algebras. In cases of representations of affine Lie algebras, Frenkel
 \cite{Fr} found a duality between different affine Lie algebras of type $A$ at
 different levels. This duality  now called  level-rank duality, establishes a  category equivalence between the category
$ \cal{C}(\widehat{\frak{sl}_n}, l)$ of representations of
$\widehat{\frak{sl}_n} $ in level $l$ and the category
$ \cal{C}(\widehat{\frak{sl}_l}, n)$.
This equivalence has been made as equivalence of
tensor categories recently in \cite{OS} and  its
appearances have been found in many fields of mathematics, such as
geometry and number theory, quantum groups,
in addition to conformal field theory. See \cite{ABS, NS,AS, Mu, MO}.
Both Theorems \ref{thm:1.1} and \ref{thm:1.2} give another form of the
level-rank duality. In particular, Theorem~
\ref{thm:1.2} can be regarded as a combination of level-rank duality and
Howe duality. Further relations with those dualities still need to be
investigated, in particular its interpretation in the geometric setting in terms of stable vector bundles over curves as in \cite{MO}.

\subsection{}  Given a semi-simple Lie algebra $\frak{g}$ and a levi subalgebra $ \frak{l}$, branching rule is to decompose irreducible $\frak{g}$-modules into direct sums of irreducible $ \frak{l}$-modules. Similarly,  one can consider the affine version of decomposing irreducible $ \wg$-modules
into direct sums of irreducible $ \hat{\frak{l}}$-modules, in the form of
\begin{equation}
L_{\wg}(l, \bar{\Lambda})=\oplus L_{\hat{\frak{l}}}(l, \bar{\Lambda}_{\frak{l}})\otimes N_{\frak{g, l}}(l, \bar{\Lambda}, \bar{\Lambda}_{\frak{l}}).
\end{equation}
In this case, $C_{L_{\wg}(l, 0)}(L_{\hat{\frak{l}}}(l, 0))
=N_{\frak{g, l}}(l, 0, 0)$ is a vertex operator algebra and all others
$N_{\frak{g, l}}(l, \bar{\Lambda}, \bar{\Lambda}_{\frak{l}})$ are the modules for this vertex operator algebra.
When $ \frak{l}=\frak{h}$ is the Cartan subalgebra of
$ \frak{g}$, $C_{L_{\wg}(l, 0)}(L_{\hat{\frak{l}}}(l, 0))$
is called the parafermion, which has been studied extensively  \cite{ZF}, \cite{CGT}, \cite{DLY2}, \cite{DLWY}, \cite{ALY}, \cite{DW1}-\cite{DW3}, etc.
For general Levi-subalgebra $\frak{l}$, we  will call
$C_{L_{\wg}(l, 0)}(L_{\hat{\frak{l}}}(l, 0))$ the relative parafermion.
Theorem~\ref{thm:1.2} says that decomposition of tensor products
of $ \widehat{\frak{sl}_{n}}$-irreducible integral representations
at various levels $ (l_1, \cdots, l_s)$ is related to the branching
role of irreducible representations of $\widehat{\frak{sl}_{l}}$
at the level $n$ with respect to the standard  Levi subalgebra defined by $(l_1, \cdots, l_s)$. In \cite{BEHHH}, it was suggested that certain unifying $\cal{W}$-algebras can be realized in terms of coset construction via embedding $\frak{sl}_{n}$ into $\frak{sl}_{n+1}$. Theorem~\ref{thm:1.2} is a more general form of this suggestion.

In fact, the left hand side corresponds to the conformal field theory  coset construction  $\f{su}(n)_{l_1}\oplus\cdots \oplus \f{su}(n)_{l_s}/\f{su}(n)_{l_1+\cdots +l_s}$ while the right hand side corresponds to the conformal field theory coset construction $ \f{su}(l_1+\cdots +l_s)_n/\f{l}(l_1, \cdots, l_s)_n$. Theorem~\ref{thm:1.2} seems to resemble a version of Howe reciprocity property of branching rules \cite{HTW}.  It seems the pairs satisfy the Gelfand pair (or Howe's symmetric pair) properties in the context of vertex operator algebras.

In \cite{H}, Howe studied a pair of reductive subgroups $(G_1, G_2)$ in a given reductive Lie group $G$ such that $G_1$ and $G_2$ are mutually centralizers and a certain representation $\pi$ of $G$ has a tensor decomposition into direct sum of the form
\[ \pi=\sum_{\tau} \tau\otimes \rho(\tau)\]
such that $ \rho: \Irr(G_1)\rightarrow \Irr(G_2)$ defines a correspondence. In  terms of vertex operator algebras, Theorem~\ref{thm:1.1} implies that $L_{\widehat{\f{sl}_n}}(l,0) $ and $K(\f{sl}_l, n)$ form a duality pair in $ V_{A_{n-1}^{\times l}}$  while $K(\f{sl}_l, n)$ and $V_{\sqrt{n} A_{l-1}}$ form a duality pair in $ L_{\widehat{\f{sl}_l}} (n,0)$. More details of correspondences of irreducible modules will be discussed in Section 5.
Theorem~\ref{thm:1.2} provides two other duality pairs. These pairs satisfy certain reciprocity law of duality pairs as discussed in \cite{howe:reciprocity}.

\subsection{} Various decompositions of  the tensor product $L_\wg(1,0)^{\otimes l}$
have been studied extensively. The commutant vertex operator algebra
$C=C_{L_\wg(1,0)^{\otimes l}}(L_\wg(l,0))$ is called coset construction. They
are expected to be rational. In case of $\frak{g}=\frak{sl}_2$, we  proved the rationality
for all $l$  in \cite{JL}. For smaller
 $l=2, 3$ there has been a lot of computational results.  In \cite{JL} all
irreducible modules for the commutant algebra have been classified.
Not all irreducible  $C$-modules appear as $M_{\wg}(l, \bar{\Lambda})$
in the decomposition \eqref{eq:decomposition}, but they appear in the
decompositions of $\otimes _{i=1}^{l}L_{\wg}(1,\bar{\Lambda}_{i})$. This is quite different from the classical Schur duality for associative algebras for which  all irreducible modules of the commutant algebra appear in the decomposition.  For $\frak{g}=\frak{sl}_{n}$ with
higher rank, rationality question is still illusive.  On the other hand, the parafermion
vertex operator algebras have been studied extensively recently and
have been expected to be rational as well \cite{DLY2}, \cite{DLWY}, \cite{ALY}, \cite{DW1}-\cite{DW3}.  Theorem~\ref{thm:1.1} explains that
cost-constructions and parafermions are the same. Establishment of
rationality of one would give the other. Theorem~\ref{thm:1.1} might provide
a way to toward answering the rationality question.

 The idea of establishing Theorems~\ref{thm:1.1} and \ref{thm:1.2} is to use the fact that
$L_{\wg}(1,0)$ is a lattice vertex operator algebra corresponding to the root
lattice $A_{n-1}$. Thus the tensor product in \eqref{eq:decomposition} is
 the lattice vertex operator algebra corresponding to the lattice
$ A_{n-1}^{\times l}$ which contains a sublattice $ N^l_n$ such that  the commutant of  the Heisenberg vertex operator algebra $L_{\widehat{\frak{h}}}(l, 0)$ in $L_{\wg}(1, 0)^{\otimes l}$ is isomorphic to the lattice vertex operator algebra $ V_{N^l_n}$ which  particularly contains $ C_{L_{\wg}(1,0)^{\otimes l}}(L_{\wg}(l,0))$. On the other hand, the commutant of $ C_{L_{\wg}(1,0)^{\otimes l}}(L_{\wg}(l,0))$ in $ V_{N^l_n}$ is the parafermion $ K(\frak{sl}_n, l)$ \eqref{de5}. A similarly defined lattice $ \tilde{N}_{l}^n \subseteq \widetilde{A}_{l-1}^{\times n}$ such that  $V_{\widetilde{N}_n^l}$ contains the parafermion $ K(\frak{sl}_l, n)$. Although the two lattices $ N_n^l$ and $\widetilde{N}_l^n$ are isomorphic, the isomorphism between $ V_{N_n^l}$ and $V_{\widetilde{N}_l^n}$ we needed is not the one coming from this lattice isomorphism. A key point is to find a correct isomorphism from $V_{N_{n}^l}$ to $V_{\widetilde{N}^n_l}$, so that $ C_{L_{\wg}(1,0)^{\otimes l}}(L_{\wg}(l,0))$ is mapped  to $K(\frak{sl}_l, n)$ in $V_{\widetilde{N}^n_l}$.


\subsection{} We now outline the paper. In Section 2, we briefly review basics on vertex operator algebras which we need in the latter sections. In Section 3, we establish Theorem~\ref{thm:1.1}.  Theorem~\ref{thm:1.2} is proved in Section 4.  In Section 5, we discuss various duality pairs and reciprocity laws in the spirit of \cite{H} and \cite{howe:reciprocity} that appear in the context of vertex operator algebras appeared in this paper. Such properties called {\em SeeSaw} property in \cite{Ku} in the context of reductive groups.  In Section 6 we summarize results  in answering the rationality  of parafermion vertex operator algebras and coset constructions.

{\em Acknowledgement:} This work is a continuation of the effort to understand the decomposition of the tensor product of basic representation in the spirit of Schur-Weyl duality. This work started when the second author was visiting SJTU during the summer of 2013. The second author thanks the support by SJTU and its hospitality. The main results of the paper were  achieved during the visit of the first author to Kansas State University in January 2014. This visit was supported by a collaborative project between the departments of Mathematics of SJTU and Kansas State University  funded by SJTU. The first author thanks C. Lam for informing her, after her talk on this paper at a conference in Dalian,  June 14-17, 2014, that he also has proved a result similar to Theorem 1.1.  

\section{Preliminaries}
\setcounter{equation}{0}

Let $V=(V,Y,{\bf 1},\omega)$ be a vertex operator algebra \cite{B},
\cite{FLM}, \cite{LL}. We review various notions of $V$-modules and the definition of rational vertex operator algebras and some basic facts (cf.
\cite{FLM}, \cite{Z}, \cite{DLM3}, \cite{DLM4}, \cite{LL}).  We also recall intertwining
operators,  fusion rules and some consequences following \cite{FHL}, \cite{ADL}, \cite{W}, \cite{DMZ}, \cite{A}, \cite{DJL}.

\begin{defn} A weak $V$-module is a vector space $M$ equipped
with a linear map
$$
\begin{array}{ll}
Y_M: & V \rightarrow {\rm End}(M)[[z,z^{-1}]]\\
 & v \mapsto Y_M(v,z)=\sum_{n \in \Z}v_n z^{-n-1},\ \ v_n \in {\rm End}(M)
\end{array}
$$
satisfying the following:

1) $v_nw=0$ for $n>>0$ where $v \in V$ and $w \in M$,

2) $Y_M( {\textbf 1},z)=\on{id}_M$,

3) The Jacobi identity holds:
\begin{eqnarray}
& &z_0^{-1}\de \left({z_1 - z_2 \over
z_0}\right)Y_M(u,z_1)Y_M(v,z_2)-
z_0^{-1} \de \left({z_2- z_1 \over -z_0}\right)Y_M(v,z_2)Y_M(u,z_1) \nonumber \\
& &\ \ \ \ \ \ \ \ \ \ =z_2^{-1} \de \left({z_1- z_0 \over
z_2}\right)Y_M(Y(u,z_0)v,z_2).
\end{eqnarray}
\end{defn}


\begin{defn}
An admissible $V$ module is a weak $V$ module  which carries a
$\Z_+$-grading $M=\bigoplus_{n \in \Z_+} M(n)$, such that if $v \in
V_r$ then $v_m M(n) \subseteq M(n+r-m-1).$
\end{defn}

\begin{defn}
An ordinary $V$ module is a weak $V$ module which carries a
$\C$-grading $M=\bigoplus_{\l \in \C} M_{\l}$, such that:

1) $\dim(M_{\l})< \infty,$

2) $M_{\l+n}=0$ for fixed $\l$ and $n<<0,$

3) $L(0)w=\l w=\wt(w) w$ for $w \in M_{\l}$ where $L(0)$ is the
component operator of $Y_M(\omega,z)=\sum_{n\in\Z}L(n)z^{-n-2}.$
\end{defn}

\begin{remark} \ It is easy to see that an ordinary $V$-module is an admissible one. If $W$  is an
ordinary $V$-module, we simply call $W$ a $V$-module.
\end{remark}

We call a vertex operator algebra rational if the admissible module
category is semisimple. We have the following result from
\cite{DLM3} (also see \cite{Z}).

\begin{theorem}\label{tt2.1}
If $V$ is a  rational vertex operator algebra, then $V$ has finitely
many irreducible admissible modules up to isomorphism and every
irreducible admissible $V$-module is ordinary.
\end{theorem}

Suppose that $V$ is a rational vertex operator algebra and let
$M^1,...,M^k$ be the irreducible  modules such that
$$M^i=\oplus_{n\geq 0}M^i_{\l_i+n}$$
where $\l_i\in\Q$ \cite{DLM4},  $M^i_{\l_i}\ne 0$ and each
$M^i_{\l_i+n}$ is finite dimensional.

A vertex operator algebra is called $C_2$-cofinite if $C_2(V)$ has
finite codimension where $C_2(V)=\<u_{-2}v|u,v\in V\>$ \cite{Z}.

\begin{remark} If $V$ is a vertex operator algebra satisfying $C_{2}$-cofinite
property, $V$ has only finitely many irreducible admissible modules
up to isomorphism \cite{DLM3}, \cite{L}, \cite{Z}.
\end{remark}

We now recall lattice vertex operator algebras and related results. Let $L$ be a positive definite even lattice in the sense that $L$ has a
symmetric positive definite  $\mathbb{Z}$-valued bilinear form $(\cdot, \cdot)$ such that $(\alpha,
\alpha)\in 2\mathbb{Z}$ for any $\alpha \in L$.
We set $\h=\C\otimes_{\Z} L$ and extend $(\cdot\,,\cdot)$ to a
$\C$-bilinear form on $\h$. Let
$\hat{\h}=\C[t,t^{-1}]\otimes\h\oplus\C C$ be the affinization of
commutative Lie algebra $\h$ defined by
\begin{align*}
[\beta_1\otimes t^{m},\,\beta_2\otimes
t^{n}]=m(\beta_1,\beta_2)\delta_{m,-n}C\hbox{ and }[C,\hat{\h}]=0
\end{align*}
for any $\beta_i\in\h,\,m,\,n\in\Z$. Then $\hat{\h}^{\geq
0}=\C[t]\otimes\h\oplus\C C$ is a commutative subalgebra. For any
$\lambda\in\h$, we can define a one dimensional $\hat{\h}^{\geq
0}$-module $\C e^\lambda$ by the actions $\rho(h\otimes
t^{m})e^\lambda=(\lambda,h)\delta_{m,0}e^\lambda$ and
$\rho(C)e^\lambda=e^\lambda$ for $h\in\h$ and $m\geq0$. Now we
denote by
\begin{align*}
M(1,{\lambda})=U(\hat{\h})\otimes_{U(\hat{\h}^{\geq 0})}\C
e^\lambda\cong S(t^{-1}\C[t^{-1}]\otimes \h),
\end{align*}
the $\hat{\h}$-module induced from $\hat{\h}^{\geq 0}$-module $\C
e^\lambda$. Set $M(1)=M(1,0).$ Then there exists a linear map
$Y:M(1)\to(\End M(1,\lambda)[[z,z^{-1}]]$ such that
$(M(1),\,Y,\,\1,\,\w)$ has a simple vertex operator algebra
structure and $(M(1,\lambda),Y)$ becomes an irreducible
$M(1)$-module for any $\lambda\in\h$ (see \cite{FLM}). The vacuum
vector and the Virasoro element are given by $\1=e^0$ and
$\w=\frac{1}{2}\sum_{a=1}^{d}h_a(-1)^2\1$ respectively, where
$\{h_a|a=1,2,\cdots,d\}$ is an orthonormal basis of $\h$.

Let $\widehat{L}$ be the canonical central extension of $L$ by
$\langle \kappa\rangle= \langle\kappa|\kappa^2= 1\rangle$ (see
\cite{FLM}):
$$1\rightarrow \langle\kappa
\rangle
\rightarrow\widehat{L}\stackrel {-}{\rightarrow} L\rightarrow 1$$
with the commutator map $c(\al,\be)=\kappa^{(\al,\be)}$ for
$\al,\be\in L$. Let $e: L\rightarrow\widehat{L}$  ($\al \mapsto e_{\al}$) be a section such
that $e(0)=1$ and $\epsilon: L\times L\rightarrow\langle \kappa\rangle $  the
corresponding 2-cocycle. We may assume that $\epsilon$ is
bimultiplicative. Then
$\epsilon(\al,\be)\epsilon(\be,\al)=\kappa^{(\al,\be)}$,
$$
\epsilon(\al,\be)\epsilon(\al+\be,\gamma)=\epsilon(\be,\gamma)(\al,\be+\gamma),$$
and $e_{\al}e_{\be}=\epsilon(\al,\be)e_{\al+\be}$ for
$\al,\be,\gamma\in L$. Let $\theta$ denote the automorphism of
$\widehat{L}$ defined by $\theta(e_{\al})=e_{-\al}$ and
$\theta(\kappa)=\kappa$. Set
$K=\{a^{-1}\theta(a)|a\in\widehat{L}\}$. Note that if $(\alpha,\beta)\in 2\Z$ for
all $\alpha, \beta\in L$ then  $\widehat{L}=L\times \langle\kappa
\rangle$ is a direct product of abelian groups. In this case we may choose
 $\epsilon(\al,\be)=1$ for all $\alpha,\beta\in L.$

The lattice vertex
operator algebra associated to $L$ is given by
$$V_L=M(1)\otimes \C^{\epsilon}[L],$$
 where
$\C^{\epsilon}[L]$ is the twisted group algebra of $L$ with a basis $e^{\alpha}$ for
$\alpha\in L$ and is an $\widehat L$-module such that $e_\alpha
e^\beta= \epsilon(\al,\be)e^{\alpha+\beta}.$ Note that if $(\al,\be)\in 2\Z$ for
all $\alpha,\beta\in L$ then $\C^{\epsilon}[L]=\C[L]$ is the usual group algebra.

Recall that
$L^{\circ}=\{\,\lambda\in\h\,|\,(\alpha,\lambda)\in\Z\,\}$
is the dual lattice of $L$.  There is an $\widehat{L}$-module
structure on $\C[L^{\circ}]=\bigoplus_{\lambda\in
L^{\circ}}\C e^\lambda$ such that $\kappa$ acts as $-1$
(see \cite{DL}). Let $L^{\circ}=\cup_{i\in
L^{\circ}/L}(L+\lambda_i)$ be the coset decomposition such that
$(\lambda_i,\lambda_i)$ is minimal among all $(\lambda,\lambda)$ for
$\lambda\in L+\lambda_i.$ In particular, $\lambda_0=0.$
Set $\C[L+\lambda_i]=\bigoplus_{\alpha\in L}\C
e^{\alpha+\lambda_i}.$ Then $\C[L^{\circ}]=\bigoplus_{i\in
L^{\circ}/L}\C[L+\lambda_i]$ and each $\C[L+\lambda_i]$ is an
$\widehat L$-submodule of $\C[L^{\circ}].$ The action of
$\widehat L$ on $\C[L+\lambda_i]$ is as follows:
$$e_{\alpha}e^{\beta+\lambda_i}=\e(\a,\b)e^{\a+\b+\l_i}$$
for $\alpha,\,\b\in L.$ Although it looks that the $\hat{L}$-module structure on each
$\C[L+\lambda_i]$ depends on the choice of $\lambda_i$ in
$L+\lambda_i,$ it is easy to check that different choices of
$\lambda_i$ give isomorphic $\widehat L$-modules.

\vskip 0.3cm
We now recall the  notion of intertwining operators and fusion
rules from [FHL].
\begin{defn}
Let $V$ be a vertex operator algebra and $M^1$, $M^2$, $M^3$ be weak $V$-modules. An intertwining
operator $\mathcal {Y}( \cdot , z)$ of type $\left(\begin{tabular}{c}
$M^3$\\
$M^1$ $M^2$\\
\end{tabular}\right)$ is a linear map$$\mathcal
{Y}(\cdot, z): M^1\rightarrow \Hom(M^2, M^3)\{z\}$$ $$v^1\mapsto
\mathcal {Y}(v^1, z) = \sum_{n\in \mathbb{C}}{v_n^1z^{-n-1}}$$
satisfying the following conditions:

(1) For any $v^1\in M^1, v^2\in M^2$ and $\lambda \in \mathbb{C},
v_{n+\lambda}^1v^2 = 0$ for $n\in \mathbb{Z}$ sufficiently large.

(2) For any $a \in V, v^1\in M^1$,
$$z_0^{-1}\delta(\frac{z_1-z_2}{z_0})Y_{M^3}(a, z_1)\mathcal
{Y}(v^1, z_2)-z_0^{-1}\delta(\frac{z_1-z_2}{-z_0})\mathcal{Y}(v^1,
z_2)Y_{M^2}(a, z_1)$$
$$=z_2^{-1}\delta(\frac{z_1-z_0}{z_2})\mathcal{Y}(Y_{M^1}(a, z_0)v^1, z_2).$$

(3) For $v^1\in M^1$, $\dfrac{d}{dz}\mathcal{Y}(v^1,
z)=\mathcal{Y}(L(-1)v^1, z)$.
\end{defn}
All of the intertwining operators of type $\left(\begin{tabular}{c}
$M^3$\\
$M^1$ $M^2$\\
\end{tabular}\right)$ form a vector space denoted by $I_V\left(\begin{tabular}{c}
$M^3$\\
$M^1$ $M^2$\\
\end{tabular}\right)$. The dimension of $I_V\left(\begin{tabular}{c}
$M^3$\\
$M^1$ $M^2$\\
\end{tabular}\right)$ is called the
fusion rule of type $\left(\begin{tabular}{c}
$M^3$\\
$M^1$ $M^2$\\
\end{tabular}\right)$ for $V$.

We have the following results from \cite{D} and \cite{DL}.
\begin{theorem}\label{Dong-L}
 Let $L$ be a positive-definite even lattice. Then

(1) $V_{L}$ is rational and $V_{L+\lambda_i}$, $i\in L^{\circ}/L$ are all the irreducible modules of $V_{L}$.

(2) For $\la,\mu,\gamma\in L^{\circ}$, $$I_{V_L}\left(\begin{tabular}{c}
$V_{L+\gamma}$\\
$V_{L+\la}$ $V_{L+\mu}$\\
\end{tabular}\right)\neq 0$$
if and only if $\gamma-\la-\mu\in L$.

\end{theorem}

\section{Level-Rank Duality}
\setcounter{equation}{0}

For $k\in\Z_{+}$ and  a complex finite-dimensional simple Lie algebra $\frak{g}$ with normalized non-degenerate bilinear form, let $\widehat{\frak g}$ be the corresponding  affine Lie algebra and $L_{\widehat{\frak g}}(k,0)$  the simple vertex operator algebra associated with the integrable highest weight module of $\widehat{\frak g}$ with level $k$. Let $\frak{h}$ be the  Cartan subalgebra of $\frak{g}$ and $L_{\widehat{\frak{h}}}(k,0)$  the Heisenberg  vertex operator subalgebra of $L_{\widehat{\frak g}}(k,0)$ generated by $\frak{h}$. Let
$$
K(\frak{g},k)=\{v\in L_{\widehat{\frak g}}(k,0)| [Y(u,z_1),Y(v,z_2)]=0, u\in L_{\widehat{\frak h}}(k,0)\}.
$$
Then  $K(L_{\widehat{\frak g}}(k,0))$ is the so-called parafermion vertex operator algebra (see \cite{BEHHH}, \cite{DLY2}, \cite{DW2}).

Let $s\in\Z_{\geq 2}$ and $\underline{\ell}=(l_1,\cdots,l_s)$ such that $l_{1},\cdots,l_{s}\in \Z_{+}$. Let $L_{\widehat{\frak g}}(l_{i},0)$ be the simple vertex operator algebra associated with the integrable highest weight module of $\widehat{\frak g}$ with level $l_{i}$, $i=1,2,\cdots,s$. Then we have the tensor product vertex operator algebra:
$$V=L_{\widehat{\frak g}}(l_{1},0)\otimes L_{\widehat{\frak g}}(l_{2},0)\otimes \cdots\otimes L_{\widehat{\frak g}}(l_{s},0).$$
Denote
$$
l=|\underline{\ell}|=\sum\limits_{i=1}^{s}l_{i}.
$$
${\frak g}$ can be naturally imbedded into the weight one subspace of $V$ diagonally as follows:
$$
{\frak{g}}\hookrightarrow V_{1}\subseteq V
$$
$$
a\mapsto a(-1){\bf 1}\otimes {\bf 1}\otimes\cdots \otimes {\bf 1}+{\bf 1}\otimes a(-1){\bf 1}\otimes {\bf 1}\otimes \cdots\otimes {\bf 1}+{\bf 1}\otimes \cdots\otimes {\bf 1}\otimes a(-1){\bf 1}.
$$
It is known that the vertex operator subalgebra $U$ of $V$ generated by $\frak{g}$ is isomorphic to the simple vertex operator algebra $L_{\widehat{\frak{g}}}(l,0)$ (see for example \cite{FZ} and \cite{LL}). Set
$$
C_{V}(U)=\{v\in V| [Y(u,z_1), Y(v,z_2)]=0, u\in U \}.
$$
Then (see \cite{LL})
$$
C_{V}(U)=\{v\in V| u_{m}v=0, u\in U, m\geq 0 \}
$$
is a vertex operator subalgebra of $V$ with a different conformal vector. We will denote $L_{\widehat{\frak{g}}}(l,\bar{\Lambda})$ the irreducible module of the vertex operator algebra $L_{\widehat{\frak{g}}}(l,0)$ associated to the  highest weight $\Lambda$ of $\widehat{\frak{g}}$. $\bar{\Lambda}=\Lambda|_{\frak{h}}$ is defined as follows ( see \S 6.2 in \cite{K}):
$$
\Lambda=\bar\Lambda+l\Lambda_{0}+m\delta,
$$
for some $m\in\Z$, where $\Lambda_{0}$ is the highest weight of $L_{\widehat{\frak{g}}}(1,0)$ as an irreducible  highest weight   $\widehat{\frak{g}}$-module.
We have the following lemma.
\begin{lem}$C_{V}(U)$ is a simple vertex operator subalgebra of $V$.
\end{lem}
\begin{proof} Since $U$ is rational, it follows that as a $U$-module $V$ is completely reducible. Then
\[
V=U\otimes C_{V}(U)\bigoplus_{\Lambda}L_{\widehat{\frak{g}}}(l,\bar{\Lambda})\otimes M(\Lambda),
\]
with the summation over $ \Lambda$ such that $\bar{\Lambda}\neq 0$ and $L_{\widehat{\frak{g}}}(l,\bar{\Lambda})$ are irreducible $L_{\widehat{\frak{g}}}(l,0)$-modules and  $M(\Lambda)$ are modules of $C_{V}(U)$. Let $W$ be a non-zero ideal of $C_{V}(U)$. Note that for any $u\in U\otimes W$, $v\in L_{\widehat{\frak{g}}}(l,\bar{\Lambda})\otimes M(\Lambda)$, and $m\in\Z$, we have $u_mv\in L_{\widehat{\frak{g}}}(l,\bar{\Lambda})\otimes M(\Lambda)$. By \cite[Cor. 4.5.10, Prop. 3.1.19]{LL}, the intersection  of $U\otimes C_{V}(U)$ and the ideal of $V$ containing $U\otimes W$ is $U\otimes W$. Since $V$ is simple, it follows that $W=C_{V}(U)$. 
\end{proof}

\vskip 0.3cm
 In the rest of  this paper, we always assume that $\frak{g}=\frak{sl}_{n}(\C)$ ($n\geq 2$) is the  simple Lie algebra of type $A_{n-1}$ over $\C$. In this section we  study the case that $s\geq 2$ and $l_{1}=\cdots=l_{s}=1$. Then $l=|\underline{\ell}|=s$ and
$$
V=L_{\widehat{\frak{sl}_{n}}}(1,0)^{\otimes l},  \ \ U\cong L_{\widehat{\frak{sl}_{n} }}(l,0).
$$

  Note that $L_{\widehat{\frak{sl}_{n} }}(1,0)\cong V_{A_{n-1}}$, where  $V_{A_{n-1}}$ is the lattice vertex operator algebra associated to the  root lattice  $A_{n-1}$.

Let
$$A_{n-1}^{\times l}=\bigoplus_{i=1}^{n-1}\bigoplus_{j=1}^{l}\Z\alpha^{ij}$$
 be a lattice such that
for  $1\leq i,p\leq n-1$, $1\leq j,q\leq l$,
$$
(\alpha^{ij},\al^{pq})=0, \ {\rm if} \ j\neq q,
$$
$$
(\alpha^{ij},\al^{pj})=\left\{
\begin{array}{ll}
 2, & \ {\rm if}  \ p=i,\\
-1, & \ {\rm if} \ i=p-1 \ {\rm or} \ i=p+1,\\
0  & \ {\rm otherwise.}
\end{array}
\right.
$$
Then the lattice vertex operator algebra $V_{A_{n-1}^{\times l}}$ is isomorphic to $V$ and the $j^{th}$ factor of $V$ is the lattice vertex operator algebra $V_{A_{n-1}^{(j)}}\cong V_{A_{n-1}}$ with $A_{n-1}^{(j)}=\bigoplus_{i=1}^{n-1}\Z\al^{ij}$.

 For $1\leq i\leq n-1$, define
$$
e^{i}=\sum\limits_{j=1}^{l}e^{\al^{ij}}, \ e^{-i}=\sum\limits_{j=1}^{l}e^{-\al^{ij}}.
$$
Since $L_{\widehat{\frak g}}(l,0)$ is generated by $e^{i}, e^{-i}$, $i=1,2,\cdots, n-1$, it follows that
$$
C_{V}(U)=\{v\in V| e^{i}_{m}v=e^{-i}_{m}v=0, 1\leq i\leq n-1, m\geq 0 \}.
$$
Let $N_{n}^{l}$ be  the sublattice
$$
 N_{n}^{l}=\bigoplus_{i=1}^{n-1}\bigoplus_{j=1}^{l-1}\Z(\al^{ij}-\al^{i,j+1})\subseteq A_{n-1}^{\times l}
$$
 (with the restriction of the bilinear form of $A_{n-1}^{\times l}$) and  $V_{N_n^{l}}$ the vertex operator algebra   associated with $N_{n}^l$. It is obvious that $V_{N_n^{l}}\subseteq V$.
 Set
  \begin{equation}\label{lattice1}
  K=\bigoplus_{i=1}^{n-1}\Z\al^i, \ \frak{h}_n=\bigoplus_{i=1}^{n-1}\C\al^{i}(-1){\bf 1},\end{equation}
  where $\al^i=\sum\limits_{j=1}^{l}\al^{ij}$, $\al^i(-1){\bf 1}=\sum\limits_{j=1}^{l}\al^{ij}(-1){\bf 1}$, $i=1,2,\cdots,n-1, j=1,2,\cdots,l$. Then $K\cong \sqrt{l}A_{n-1}$.

  Let $L_{\widehat{{\frak h}_n}}(l,0)\subseteq L_{\widehat{\sl_{n}}}(l,0)$ be a Heisenberg vertex operator subalgebra  of $V$ generated by $\frak{h}_n$. Then  as an $L_{\widehat{\frak{h}_n}}(l,0)$-module, $V$ is completely reducible. Note that $V$ is linearly spanned by
$$
\gamma^1(-m_{1})\cdots \gamma^k(-m_k)\otimes e^{\al}, \ \gamma^1,\cdots,\gamma^k, \al\in A_{n-1}^{\times l}, k\geq 0, m_{1}\geq m_{2}\geq \cdots\geq m_{k}\geq 1,
$$
and $(\sum\limits_{j=1}^{l}\al^{ij}, \al^{pq}-\al^{p,q+1})=0$ for any $1\leq i,p\leq n-1$, $1\leq q\leq l$. Note also that $$\C\otimes_{\Z} A_{n-1}^{\times l}=\C\otimes_{\Z}K\oplus\C\otimes_{\Z}N_{n}^{l}.$$
Then it is easy to see that
 as an $\widehat{\frak{h}}_n$-module, $V$ is a direct sum of highest weight modules of $\widehat{\frak{h}}_n$ with highest weight vectors
$$
h^1(-m_{1})\cdots h^k(-m_k)\otimes e^{\al}, \ h^1,\cdots,h^k\in N_{n}^{l}, \al\in A_{n-1}^{\times l}, k\geq 0, m_{1}\geq m_{2}\geq \cdots\geq m_{k}\geq 1.
$$
Then we have
\begin{equation}\label{para3}
V_{N_{n}^l}=\{v\in V\;|\; \sum\limits_{j=1}^{l}\al^{ij}(m)v=0, 1\leq i\leq n-1, m\geq 0\}.\end{equation}
So we get the following lemma.
\begin{lem}
$C_{V}(L_{\widehat{{\frak h}_{n}}}(l,0))=V_{N_{n}^{l}}$. In particular, $C_{V}(U)\subseteq V_{N_n^{l}}.$
\end{lem}

\vskip 0.3cm
For $1\leq i\leq n-1$, we have
$$
e^{i}_{-1}e^{-i}
=\frac{1}{2}\sum\limits_{j=1}^{l}\al^{ij}(-2){\bf 1}+\frac{1}{2}\sum\limits_{j=1}^{l}\al^{ij}(-1)^2{\bf 1}+\sum\limits_{1\leq p <q\leq l}(e^{\al^{ip}-\al^{iq}}+e^{-\al^{ip}+\al^{iq}}).
$$
Note that
$$
\sum\limits_{j=1}^{l}\al^{ij}(-1)^2{\bf 1}=\frac{1}{l}\sum\limits_{1\leq p<q\leq l}(\al^{ip}-\al^{iq})(-1)^{2}{\bf 1}+\frac{1}{l}(\al^{i1}+\cdots +\al^{il})(-1)^{2}{\bf 1}.
$$
So on $V_{N_n^{l}}$ we have
$$
(e^{i}_{-1}e^{-i})_m=\sum\limits_{1\leq p<q\leq l}[\frac{1}{2l}(\al^{ip}-\al^{iq})(-1)^{2}{\bf 1}+(e^{\al^{ip}-\al^{iq}}+e^{-\al^{ip}+\al^{iq}})]_{m}, \ m\geq 0.
$$
For $1\leq i\leq n-1$, set
\begin{equation}\label{eq3.1}
'{\omega}^{i}=\frac{1}{l+2}\sum\limits_{1\leq p<q\leq l}[\frac{1}{2l}(\al^{ip}-\al^{iq})(-1)^{2}{\bf 1}+(e^{\al^{ip}-\al^{iq}}+e^{-\al^{ip}+\al^{iq}})].
\end{equation}
The following lemma can be checked directly.
\begin{lem}
For $1\leq i\leq  n-1$, $'{\omega}^{i}$ is a Virasoro vector of $V$ with central charge $\frac{2(l-1)}{l+2}$.
\end{lem}

Let $K$ be defined as in (\ref{lattice1}) and  $V_{K}$  the lattice vertex operator  algebra associated to the lattice $K$. Since
  $$(e^{\al_{i1}}+\cdots+e^{\al_{il}})^{l-1}_{-1}(e^{\al_{i1}}+\cdots+e^{\al_{il}})=l!e^{\al^{i1}+\cdots+\al^{il}}=l!e^{\al^i}, \ i=1,2,\cdots,n-1,$$
 it follows that $e^{\al^i}\in L_{\widehat{\sl_{n}}}(l,0)$, $1\leq i\leq n-1$. Similarly, $e^{-\al^i}\in L_{\widehat{\sl_{n}}}(l,0)$, $1\leq i\leq n-1$. Therefore $V_{K}$ is a vertex operator subalgebra of $L_{\widehat{\sl_{n}}}(l,0)$. Hence every $L_{\widehat{\sl_n}}(l,0)$-module regarded as a $V_{K}$-module is completely reducible. By Lemma \ref{Dong-L}, every irreducible module of $V_{K}$ comes from $V_{K+\gamma},\gamma\in K^{\circ}/K$. It is known that  the lowest weight of $V_{K+\gamma}$ for $0\neq \gamma\in K^{\circ}/K$ is positive. Note also that $\{v\in V_{K}|h_{m}v=0, h\in \frak{h}_{n},m\geq 0\}=\C{\bf 1}$. Then  we have
 \begin{equation}\label{de11}
 K(\sl_n,l)=\{v\in L_{\widehat{\sl_n}}(l, 0)|u_{m}v=0, u\in V_{K}, m\geq 0\}
 \end{equation}
 and
 \begin{equation}\label{decomp1}
 L_{\widehat{\sl_n}}(l, 0)=\bigoplus_{\gamma\in K^{\circ}/K}V_{K+\gamma}\otimes U^{(l,n)}(0,\gamma),
 \end{equation}
 where $U^{(l,n)}(0,0)=K(\sl_n,l)$ and $U^{(l,n)}(0,\gamma)$ for $\gamma\neq 0$ are modules of $K(\sl_n,l)$ (also see \cite{DW3}).
 More generally, let $L_{\widehat{\sl_n}}(l, \bar{\Lambda})$ be  an irreducible module of the vertex operator algebra  $L_{\widehat{\sl_{n}}}(l,0)$. Then we can decompose $L_{\widehat{\sl_n}}(l, \bar{\Lambda})$ as follows:
 \begin{equation}\label{decomp2}
 L_{\widehat{\sl_n}}(l, \bar{\Lambda})=\bigoplus_{\gamma\in K^{\circ}/K}V_{K+\bar{\Lambda}+\gamma}\otimes U^{(l,n)}(\bar{\Lambda},\gamma),
 \end{equation}
 where $U^{(l,n)}(\bar{\Lambda},\gamma)$ for $\gamma\in K^{\circ}$ are modules of $K(\sl_n,l)$. We have the following lemma.
 \begin{lem}\label{irr-1}
(1) In (\ref{decomp1}) and (\ref{decomp2}), if $U^{(l,n)}(\bar{\Lambda},\gamma)\neq 0$, then $U^{(l,n)}(\bar{\Lambda},\gamma)$ is an irreducible $K(\sl_n,l)$-module.

(2) $U^{(l,n)}(0,0)$ is not isomorphic to $U^{(l,n)}(\bar{\Lambda}, -\bar{\Lambda})$ as $K(\sl_n,l)$-modules, for any $\Lambda\neq 0$.
\end{lem}
\begin{proof} For $\gamma\in K^{\circ}/K$ such that $U^{(l,n)}(\bar{\Lambda},\gamma)\neq 0$, let $W$ be a non-zero $K(\sl_n,l)$-submodule of $U^{(l,n)}(\bar{\Lambda},\gamma)$. Let $M$ be the $L_{\widehat{\sl_n}}(l,0)$-submodule generated by $W$. For $\gamma_{1}\in  K^{\circ}/K$, let $u\in V_{\bar{\Lambda}+\gamma_{1}}\otimes U^{(l,n)}(\bar{\Lambda},\gamma_{1})$, $v\in V_{\bar{\Lambda}+\gamma}\otimes W$. Then by (2) of Lemma \ref{Dong-L},
$$
u_{m}v\in V_{2\bar{\Lambda}+\gamma_{1}+\gamma}\otimes U^{(l,n)}(2\bar{\Lambda}+\gamma_{1}+\gamma), \ m\in\Z.
$$
Therefore
$
u_{m}v\in V_{\bar{\Lambda}+\gamma}\otimes U^{(l,n)}(\bar{\Lambda}+\gamma)$, for $\ m\in\Z
$
if and only if $\bar{\Lambda}+\gamma_{1}\in K$. That is, $V_{\bar{\Lambda}+\gamma_{1}}=V_{K}$. Then we have $u_{m}v\in V_{\bar{\Lambda}+\gamma}\otimes W$. Hence
$$M\cap (V_{\bar{\Lambda}+\gamma}\otimes U^{(l,n)}(\bar{\Lambda}+\gamma))=V_{\bar{\Lambda}+\gamma}\otimes W.$$ Since $L_{\widehat{\sl_n}}(l,\bar{\Lambda})$ is simple, it follows that $W=U^{(l,n)}(\bar{\Lambda},\gamma)$. (1) is proved.

Let $L_{\widehat{\sl_n}}(l,\bar{\Lambda})$ be an irreducible module of the vertex operator algebra  $L_{\widehat{\sl_n}}(l,0)$ such that $\Lambda\neq 0$. Then the lowest weight of  $L_{\widehat{\sl_n}}(l,\bar{\Lambda})$ is positive ( see \cite{LL} and also Lemma 12.8 in \cite{K}). Since $U^{(l,n)}(\bar{\Lambda},-\bar{\Lambda})$ as a $K(\sl_n,l)$-module  has the same lowest weight with  $L_{\widehat{\sl_n}}(l,\bar{\Lambda})$, it follows that $U^{(l,n)}(0,0)$ can not be isomorphic to $U^{(l,n)}(\bar{\Lambda},-\bar{\Lambda})$, since the lowest weight of $U^{(l,n)}(0,0)$ as a $K(\sl_n,l)$-module is zero.
\end{proof}

Note that as an $\sl_{n}$-module, $V$ is integrable and  the weight system of $V$ is a subset  of the root lattice of $\sl_{n}$. Then for each irreducible $\sl_{n}$-submodule $W$ of $V$, there is a non-zero vector $u$ such that $h_{0}u=0$ for all $h\in \frak{h}_{n}$.
  On the other hand, since  $U=L_{\widehat{\sl_{n}}}(l,0)$ is rational, it follows that as an $L_{\widehat{\sl_{n}}}(l,0)$-module, $V$ is completely reducible. So
we have
\begin{equation}\label{de2}
V=L_{\widehat{\sl_{n}}}(l,0)\otimes C_{V}(U)\bigoplus(\bigoplus_{\la\in P^{n}_{+},\la\neq l\Lambda_{0}}L_{\widehat{\sl_{n}}}(l,\bar{\lambda})\otimes M^{(l,n)}(\bar{\lambda})),
\end{equation}
where  $P(l\Lambda_0)$ is the set of weights of $L_{\widehat{\sl_{n}}}(l,0)$ and
$$P^n_{+}=\{\la\in P(l\Lambda_{0})| \la+\delta\notin P(l\Lambda_{0}), \ (\la,\al_{i})\in\Z_{\geq 0}, i=0,1,\cdots,n-1\}.$$
 $\al_{0},\cdots,\al_{n-1}$ are the simple roots of the affine Lie algebra $\widehat{\sl_{n}}$, $\la=l\Lambda_{0}+\bar{\la}-m\delta$ for some $m\geq 0$ as introduced in $\S 6.2 $ in \cite{K}, and $(\Lambda_{0},\al_{i})=\delta_{i0}$, $i=0,1,\cdots,n-1$.
 $M^{(l,n)}(\bar{\lambda})\subseteq V_{N_n^{l}}$ are modules of $C_{V}(U)$, $\la\in P^n_{+}$. As above  considering  $V$ as a module of the Heisenberg vertex operator algebra $L_{\widehat{\frak{h}_{n}}}(l,0)$, together with (\ref{de2}) we have
\begin{equation}\label{de3}
V_{N_{n}^l}=K(\sl_{n},l)\otimes C_{V}(U)\bigoplus(\bigoplus_{\la\in P_{+}^n,\la\neq l\Lambda_{0}}W^{(l,n)}(\bar{\lambda})\otimes M^{(l,n)}(\bar{\lambda})),
\end{equation}
where we denote $W^{(l,n)}(\bar{\lambda})=U^{(l,n)}(\bar{\la},-\bar{\la})$. Following the similar argument for (\ref{de11}), we have
\begin{equation}\label{de6}
U^{(l,n)}(\bar{\la},-\bar{\la})=\{v\in L_{\widehat{\sl_{n}}}(l,\bar{\la})|h_{m}v=0, h\in{\frak h}_n\}.
\end{equation}
Then by Lemma \ref{irr-1}, $W^{(l,n)}(\bar{\lambda})$ is an irreducible $K(\sl_n,l)$-module.
We have the following lemma.
\begin{lem}\label{de-pa1}
For $\la\in P^n_{+}$, $\la\neq l\Lambda_{0}$, let $M^{(l,n)}(\bar{\lambda})$ be as in (\ref{de2}). If $M^{(l,n)}(\bar{\lambda})\neq 0$, then as a $C_{V}(U)$-module, the lowest weight of $M^{(l,n)}(\bar{\lambda})$ is a positive number.
\end{lem}
\begin{proof} For $\la\in P^{n}_{+}$, by Proposition 11.4 in \cite{K}, we have
$$
(l\Lambda_{0}+\rho,l\Lambda_{0}+\rho)-(\la+\rho,\la+\rho)\geq 0,
$$
 and the equality holds if and only if $\la=l\Lambda_{0}$, where  $(\Lambda_{0},\Lambda_{0})=0$, $(\rho,\al_{i})=1, i=0,1,\cdots,n-1, (\rho, \Lambda_{0})=0$ (see $\S 6.2$ in \cite{K}). Then we have for $\la\in P^{n}_{+}$ such that  $\la< l\Lambda_{0}$,
$$(\la+2\rho, \la)<0.
$$
Let $\omega'$ be the conformal vector of $L_{\widehat{\sl_{n}}}(l,0)$, and 
$\omega''=\omega-\omega'$  the conformal vector of $C_{V}(U)$. 
For $\la\in P^{n}_{+}$, $\la\neq l\Lambda_{0}$, denote 
$$M^{(l,n)}(\bar{\lambda})=\bigoplus_{\substack{m\in {\mathbb Q}\\ m\geq a(\bar{\lambda})}}M^{(l,n)}(\bar{\lambda})_{(m)}$$ such that $$\omega''_{1}(M^{(l,n)}(\bar{\lambda})_{(m)})=m{\rm id}|_{M^{(l,n)}(\bar{\lambda})_{(m)}}, $$
where $a(\bar{\lambda})$ is a rational number. Then by Corollary 12.8 in \cite{K} we have
$$
a(\bar{\la})=-\frac{(\la+2\rho, \la)}{2(l+n)}>0.
$$ 
\end{proof}

Let
$$
\widetilde{A}_{l-1}^{\times n}=\bigoplus_{i=1}^{l-1}\bigoplus_{j=1}^{n}\Z\be^{ij}
$$
be an even lattice such that  for  $1\leq i,p\leq l-1$, $1\leq j,q\leq n$,
$$
(\be^{ij},\be^{pq})=0, \ {\rm if} \ j\neq q,
$$
$$
(\be^{ij},\be^{pj})=\left\{
\begin{array}{ll}
 2, & \ {\rm if}  \ p=i,\\
-1, & \ {\rm if} \ i=p-1 \ {\rm or} \ i=p+1,\\
0  & \ {\rm otherwise.}
\end{array}
\right.
$$
Set
$$
\widetilde{N}_{l}^{n}=\bigoplus_{i=1}^{l-1}\bigoplus_{j=1}^{n-1}\Z(\be^{ij}-\be^{i,j+1}).
$$
Then
\begin{equation}\label{para1}
V_{\widetilde{N}_{l}^n}=\{v\in V_{\widetilde{A}_{l-1}^{\times n}}| \sum\limits_{j=1}^{n}\be^{ij}(m)v=0, 1\leq i\leq l-1, m\geq 0\}.
\end{equation}
For $1\leq i\leq l-1$, let
$$
\widetilde{\omega}^{i}=\frac{1}{n+2}\sum\limits_{1\leq p<q\leq n}[\frac{1}{2n}(\be^{ip}-\be^{iq})(-1)^{2}{\bf 1}\\
 +(e^{\be^{ip}-\be^{iq}}+e^{-\be^{ip}+\be^{iq}})],
$$
$$
\begin{array}{ll}
\widetilde{W}^{3,i}=& \sum\limits_{1\leq p,q,r\leq n}(\be^{ip}-\be^{iq})(-1)^2(\be^{ip}-\be^{ir})(-1){\bf 1}\\
& -3n\sum\limits_{1\leq q,r\leq n,q\neq r}[\sum\limits_{1\leq p\leq n,p\neq q}(\be^{ip}-\be^{iq})(-1)+\sum\limits_{1\leq p\leq n,p\neq r}(\be^{ip}-\be^{ir}))(-1)]e^{\be^{iq}-\be^{ir}}.
\end{array}
$$

 The following lemma follows  from \cite{DLY2}, \cite{DLWY} and \cite{DW1}.
\begin{lem}\label{DLWY} (1) For every  $i$ such that $1\leq i\leq l-1$, $\widetilde{W}^{3,i}$ generates a simple vertex operator algebra isomorphic to the parafermion vertex operator algebra $K(\frak{sl}_{2},n)$ with $\widetilde{\omega}^{i}$ as its Virasoro vector.

(2)  The vertex operator algebra generated by $\widetilde{W}^{3,i}$, $1\leq i\leq l-1$ is isomorphic to the parafermion vertex operator algebra $K(\frak{sl}_{l},n)$.
\end{lem}

For $1\leq i\leq l-1$, set
$$
\begin{array}{ll}
\omega^{i}=& \frac{1}{n+2}\sum\limits_{1\leq p\leq q\leq n-1}[\frac{1}{2n}(\al^{pi}-\al^{p,i+1}+\al^{p+1,i}-\al^{p+1,i+1}+\cdots+\al^{qi}-\al^{q,i+1})(-1)^{2}{\bf 1}\\
& +(-1)^{p-q+1}(e^{\al^{pi}-\al^{p,i+1}+\cdots+\al^{qi}-\al^{q,i+1}}+e^{-(\al^{pi}-\al^{p,i+1}+\cdots+\al^{qi}-\al^{q,i+1})})],
\end{array}
$$
$$
\begin{array}{ll}
W^{3,i}=-&\sum\limits_{p=1}^{n}[\sum\limits_{q=1}^{p-1}(\al^{qi}-\al^{q,i+1}+\al^{q+1,i}-\al^{q+1,i+1}+\cdots+\al^{p-1,i}-\al^{p-1,i+1})(-1)^2\\
& +\sum\limits_{q=p}^{n-1}(\al^{pi}-\al^{p,i+1}+\cdots+\al^{qi}-\al^{qi+1})(-1)^2][-\sum\limits_{k=1}^{p-1}k(\al^{ki}-\al^{k,i=1})(-1){\bf 1}\\
&+\sum\limits_{k=p}^{n-1}(n-k)(\al^{ki}-\al^{kj}(-1){\bf 1}]+3n\sum\limits_{q=1}^{n}\sum\limits_{r=1}^{q-1}(-1)^{q-r}[\sum\limits_{k=1}^{q-1}k(\al^{ki}-\al^{k,i+1})(-1)\\
& -\sum\limits_{k=q}^{n-1}(n-k)(\al^{ki}-\al^{k,i+1})(-1)+\sum\limits_{k=1}^{r-1}k(\al^{ki}-\al^{k,i+1})(-1)\\
&-\sum\limits_{k=r}^{n-1}(n-k)(\al^{ki}-\al^{k,i+1})(-1)]e^{-\sum\limits_{k=r}^{q-1}(\al^{ki}-\al^{k,i+1})}\\
&+3n\sum\limits_{q=1}^{n}\sum\limits_{r=q+1}^{n}(-1)^{q-r}[\sum\limits_{k=1}^{q-1}k(\al^{ki}-\al^{k,i+1})(-1)
-\sum\limits_{k=q}^{n-1}(n-k)(\al^{ki}-\al^{k,i+1})(-1)\\
& +\sum\limits_{k=1}^{r-1}k(\al^{ki}-\al^{k,i+1})(-1)
-\sum\limits_{k=r}^{n-1}(n-k)(\al^{ki}-\al^{k,i+1})(-1)]e^{\sum\limits_{k=q}^{r-1}(\al^{ki}-\al^{k,i+1})}.
\end{array}
$$

It is known that for $l=2$, $C_{V}(U)$ is isomorphic to the parafermion vertex operator algebra $K(\frak{sl}_{2},n)$ (see \cite{CGT} and the introduction in \cite{DLY1}).
  The following lemma  gives a concrete description of the conformal vector and the weight 3 generator  of $C_{V}(U)$  for $l=2$.
\begin{lem}\label{de-pa3}
Let $l=2$. Then
the vertex operator algebra $C_{V}(U)$ is generated by $W^{3,1}$, and $\omega^{1}$ is the conformal vector of $C_{V}(U)$.
\end{lem}
\begin{proof} Consider  the lattice vertex operator algebras $V_{\widetilde{A}_{1}^{\times n}}$ and  $V_{\widetilde{N}_{2}^{n}}$. Note that $V_{\widetilde{A}_{1}^{\times n}}$ is actually ${V_{A_{1}}}^{\otimes n}$.
By  Theorem 4.2 in \cite{LY}    the commutant $C_{V_{\widetilde{A}_{1}^{\times n}}}(L_{\widehat{\frak{sl}_{2}}}(n,0))$ of $L_{\widehat{\frak{sl}_{2}}}(n,0)$ in $V_{\widetilde{A}_{1}^{\times n}}$ is contained in $V_{\widetilde{N}_{2}^{n}}$ and the commutant of $C_{V_{\widetilde{A}_{1}^{\times n}}}(L_{\widehat{\frak{sl}_{2}}}(n,0))$ in $V_{\widetilde{N}_{2}^{n}}$ is the parafermion vertex operator  algebra $K(\frak{sl}_{2},n)$ generated by $\widetilde{\omega}^{1}$ and $\widetilde{W}^{3,1}$. Furthermore,
$$
V_{\widetilde{N}_{2}^{n}}=C_{V_{\widetilde{A}_{1}^{\times n}}}(L_{\widehat{\frak{sl}_{2}}}(n,0))\otimes K(\frak{sl}_{2},n)\bigoplus
(\bigoplus_{\substack{2\leq k_{n-1}\leq n\\ k_{n-1}\in 2\Z}}M(k_{n-1})\otimes W(k_{n-1})),
$$
where $M(k_{n-1}), 2\leq k_{n-1}\leq n$ are irreducible modules of $C_{V_{\widetilde{A}_{1}^{\times n}}}(L_{\widehat{\frak{sl}_{2}}}(n,0))$ and $W(k_{n-1}),2\leq k_{n-1}\leq n$ are irreducible modules of $K(\sl_{2},n)$.
It was proved in \cite{LS} that $C_{V_{\widetilde{A}_{1}^{\times n}}}(L_{\widehat{\frak{sl}_{2}}}(n,0))$ is generated by
\begin{equation}\label{e-lem-1}
u^{ij}=\frac{1}{16}(\be^{1i}-\be^{1,j+1})(-1)^2{\bf 1}-\frac{1}{4}(e^{\be^{1i}-\be^{1,j+1}}+e^{-\be^{1i}+\be^{1,j+1}}), \ 1\leq i\leq j\leq n-1.
\end{equation}
For $1\leq i\leq j\leq n-1$, set
$$
v^{ij}=\frac{1}{16}(\sum\limits_{k=i}^{j}(\al^{k1}-\al^{k2}))(-1)^2{\bf 1}-\frac{1}{4}(e^{\sum\limits_{k=i}^{j}(\al^{k1}-\al^{k2})}+e^{\sum\limits_{k=i}^{j}(-\al^{k1}+\al^{k2})}).
$$

Note that both $N_{n}^{2}=\bigoplus_{i=1}^{n-1}\Z(\al^{i1}-\al^{i2})$ and $\widetilde{N}_{2}^{n}=\bigoplus_{i=1}^{n-1}\Z(\be^{1i}-\be^{1,i+1})$ are  lattice of type $\sqrt{2}A_{n-1}$ with bases $\{\al^{i1}-\al^{i2}\;|\;i=1,2,\cdots,n-1\}$ and $\{\be^{1i}-\be^{1,i+1}\;|\;1=1,2,\cdots,n-1\}$ respectively. It follows that $V_{N_{n}^2}$ and $V_{\widetilde{N}_{2}^{n}}$ are isomorphic and the vertex operator algebra generated by $v^{ij}$, $1\leq i\leq j\leq n-1$ is isomorphic to the vertex operator algebra generated by $u^{ij}$, $1\leq i\leq j\leq n-1$
with the isomorphism $\sigma$: $v^{ij}\mapsto u^{ij}$, $1\leq i\leq j\leq n-1$. Let $\theta$ be an automorphism of the vertex operator algebra $V_{N_2^{n}}$  defined by
$$
\theta(e^{\al^{k1}-\al^{k2}})=-e^{-\al^{k1}+\al^{k2}}, \ \theta(e^{-\al^{k1}+\al^{k2}})=-e^{\al^{k1}-\al^{k2}}, $$$$
\theta((\al^{k1}-\al^{k2})(-1){\bf 1})=(-\al^{k1}+\al^{k2})(-1){\bf 1}, 1\leq k\leq n-1.
$$
Then
\begin{equation}\label{addition1} \theta(\sigma^{-1}(u^{ii}))=\theta(v^{ii})= {'\omega}^{i},  \ i=1,2,\cdots,n-1,
\end{equation} $$
\theta(\sigma^{-1}(\widetilde{\omega}^{1}))=\omega^{1}, \ \theta(\sigma^{-1}(\widetilde{W}^{3,1}))=W^{3,1},
$$
where $'\omega^{i}$ is defined by (\ref{eq3.1}) for $l=2$. Then we see that $C_{V}(U)=\theta(\sigma^{-1}( K(\sl_{2},n)))$ and $C_{V}(U)$ is generated by $\omega^{1}$ and $W^{3,1}$. 
\end{proof}

\vskip 0.3cm
Since $U\cong L_{\widehat{\sl_{n}}}(2,0)$ for $l=2$, it follows from (\ref{eq3.1}) that $'\omega^{i}, 1\leq i\leq n-1$ are in $K(\sl_{n},2)$. By Lemma \ref{DLWY} the vertex operator subalgebra of $U$ generated by $'\omega^{i}, 1\leq i\leq n-1$ is isomorphic to the parafermion vertex operator algebra $K(\sl_{n},2)$. By Proposition 4.6 in \cite{LS} ( also see \cite{JL}), $C_{V_{\widetilde{A}_{1}^{\times n}}}(L_{\widehat{\sl_{2}}}(n,0))$ is
generated by $u^{ii}$, $i=1,2,\cdots,n-1$. Then using the maps $\theta^{-1}$ and $\sigma$, we obtain the following lemma.
\begin{lem}\label{lr-dual1}
$C_{V_{\widetilde{A}_{1}^{\times n}}}(L_{\widehat{\frak{sl}_{2}}}(n,0))$ is isomorphic to the parafermion vertex operator algebra $K(\sl_{n},2)$, and $$
V_{\widetilde{N}_{2}^{n}}=K(\sl_{n},2)\otimes K(\frak{sl}_{2},n)\bigoplus
(\bigoplus_{\substack{2\leq k_{n-1}\leq n \\ k_{n-1}\in 2\Z}}M(k_{n-1})\otimes W(k_{n-1})).
$$
\end{lem}

\vskip0.3cm
For $n,l\in\Z_{\geq 3}$, let $W(0)$ be the vertex operator subalgebra of $V_{N_{n}^{l}}$ generated by $\omega^{i}$ and $W^{3,i}$, $1\leq i\leq l-1$ and $\widetilde{W}(0)$ the vertex operator subalgebra of $V_{\widetilde{N}_{l}^{n}}$ generated by $\widetilde{\omega}^{i}$ and $\widetilde{W}^{3,i}$, $1\leq i\leq l-1$. Then we have the following lemma.
\begin{lem}\label{de-pa2}
(1) There is a vertex operator algebra  isomorphism $\tau$ from $V_{{N}_{n}^{l}}$ to $V_{\widetilde{N}_{l}^{n}}$ such that $\tau(W(0))=\widetilde{W}(0)$.

 (2) Both $W(0)$ and $\widetilde{W}(0)$ are isomorphic to the parafermion vertex operator algebra $K(\sl_{l},n)$.
\end{lem}

\begin{proof} It is easy to check that for $1\leq i,j\leq n-1, 1\leq p,q\leq l-1$,
$$
(\al^{ip}-\al^{i,p+1}, \al^{jq}-\al^{j,q+1})=(\be^{pi}-\be^{p,i+1}, \be^{qj}-\be^{q,j+1}).
$$
Then there is a vertex operator algebra isomorphism $\tau^1$ from $V_{{N}_{n}^{l}}$ to $V_{\widetilde{N}_{l}^{n}}$ such that
$$
\tau^1((\al^{ip}-\al^{i,p+1})(-1){\bf 1})=(\be^{pi}-\be^{p,i+1})(-1){\bf 1}, \ 1\leq i\leq n-1,\ 1\leq p\leq l-1
$$
and
$$
\tau^1(e^{\al^{ip}-\al^{i,p+1}})=e^{\be^{pi}-\be^{p,i+1}}, \ 1\leq i\leq n-1,\ 1\leq p\leq l-1.
$$
Let $\theta$ be an automorphism of $V_{N_{n}^{l}}$ defined by
$$
\theta(e^{\al^{ip}-\al^{i,p+1}})=-e^{-\al^{ip}+\al^{i,p+1}}, \ \theta(e^{-\al^{ip}+\al^{i,p+1}})=-e^{\al^{ip}-\al^{i,p+1}}, $$$$
\theta((\al^{ip}-\al^{i,p+1})(-1){\bf 1})=(-\al^{ip}+\al^{i,p+1})(-1){\bf 1}, 1\leq i\leq n-1, \ 1\leq p\leq l-1.
$$
Then $\tau=\tau^1\theta$ is the desired vertex operator algebra isomorphism from $V_{{N}_{n}^{l}}$ to $V_{\widetilde{N}_{l}^{n}}$. By Lemma \ref{DLWY}, $\widetilde{W}(0)$ is isomorphic to the parafermion vertex operator algebra $K(\sl_{l},n)$. Since $\widetilde{W}(0)=\tau(W(0))$, it follows that $W(0)$ is also isomorphic to $K(\sl_{l},n)$. \end{proof}

\vskip 0.3cm
We are now in a position to state the following level-rank duality which characterizes relations between tensor decomposition and parafermion vertex operator algebras.
\begin{theorem}\label{lr-dual2}
For $n,l\in\Z_{\geq 2}$, $V=L_{\widehat{\sl_{n}}}(1,0)^{\otimes l}$ and $U=L_{\widehat{\sl_{n}}}(l,0)$, we have
$$
V=L_{\widehat{\sl_{n}}}(l,0)\otimes K(\sl_{l},n)\bigoplus(\bigoplus_{\la\in P_{+}^n,\la\neq l\Lambda_{0}}L_{\widehat{\sl_{n}}}(l,\bar{\lambda})\otimes M^{(l,n)}(\bar{\lambda})),
$$
$$
V_{{N}_{n}^{l}}=K(\sl_{n},l)\otimes K(\frak{sl}_{l},n)\bigoplus(\bigoplus_{\la\in P^n_{+},\la\neq l\Lambda_{0}}W^{(n,l)}(\bar{\lambda})\otimes M^{(l,n)}(\bar{\lambda})),
$$
where if $M^{(l,n)}(\bar{\lambda})\neq 0$, then $M^{(l,n)}(\bar{\lambda})$ as a $K(\sl_l,n)$-module is isomorphic to $W^{(n,l)}(\bar{\mu})$  for some $\mu\in P^{l}_{+}$. $P(n\Lambda_0)$ is the set of weights of $L_{\widehat{\sl_{l}}}(n,0)$ and
$$P^l_{+}=\{\la\in P(n\Lambda_{0})\;|\; \la+\delta\notin P(n\Lambda_{0}), \ (\la,\be_{i})\in\Z_{\geq 0}, i=0,1,\cdots,l-1\},$$
 where $\be_{0},\cdots,\be_{l-1}$ are the simple roots of the affine Lie algebra $\widehat{\sl_{l}}$.
\end{theorem}
\begin{proof} By Lemma \ref{lr-dual1}, we may assume that $n,l\geq 3$. Let $W(0)$ and $\widetilde{W}(0)$ be defined as above. By Lemma \ref{de-pa3}, it is easy to see that $\omega^{i},W^{3,i}\in C_{V}(U)$, $1\leq i\leq l-1$.
 Then by Lemma \ref{de-pa2}, $K(\sl_{l},n)\cong W(0) \subseteq C_{V_{{A}_{n-1}^{\times l}}}(L_{\widehat{\sl_{n}}}(l,0))$. By (\ref{de2}) and (\ref{de3}), we have
$$
V=L_{\widehat{\sl_{n}}}(l,0)\otimes C_{V}(U)\bigoplus(\bigoplus_{\la\in P^{n}_{+},\la\neq l\Lambda_{0}}L_{\widehat{\sl_{n}}}(l,\bar{\lambda})\otimes M^{(l,n)}(\bar{\lambda})),
$$
$$
V_{N_{n}^l}=K(\sl_{n},l)\otimes C_{V}(U)\bigoplus(\bigoplus_{\la\in P_{+}^n,\lambda\neq l\Lambda_{0}}W^{(l,n)}(\lambda)\otimes M^{(l,n)}(\bar{\lambda})).
$$
Regarding $V_{N_n^{l}}$ as a $C_{V}(U)$-module, by Lemma \ref{de-pa1},
 \begin{equation}\label{de5}
 V_{N_n^{l}}(0)=\{v\in V_{N_n^{l}}|\omega''_{1}v=0\}=K(\sl_{n},l),
 \end{equation}
where $\omega''$ is the conformal vector of $C_{V}(U)$. Since $K(\sl_{l},n)\cong W(0)$ is a conformal vertex operator subalgebra of $C_{V}(U)$, it follows that regarding $V_{N_n^{l}}$ as a $W(0)$-module, we also have $V_{N_n^{l}}(0)=K(\sl_{n},l)$.

Consider $V_{\widetilde{A}_{l-1}^{\times n}}$ and  $V_{\widetilde{N}_{l}^{n}}$ similarly,  we have $K(\sl_{n},l)\subseteq C_{V_{\widetilde{A}_{l-1}^{\times n}}}(L_{\widehat{\sl_{l}}}(n,0))$ and
$$
 V_{\widetilde{A}_{l-1}^{\times n}}=C_{V_{\widetilde{A}_{l-1}^{\times n}}}(L_{\widehat{\sl_{l}}}(n,0))\otimes L_{\widehat{\sl_{l}}}(n,0)
\bigoplus(\bigoplus_{\mu\in P^l_{+},\mu\neq n\Lambda_{0}}M^{(n,l)}(\bar{\mu})\otimes L_{\widehat{\sl_{l}}}(n,\bar{\mu})),
$$
\begin{equation}\label{de4}
V_{\widetilde{N}_{l}^n}=C_{V_{\widetilde{A}_{l-1}^{\times n}}}(L_{\widehat{\sl_{l}}}(n,0))\otimes \widetilde{W}(0)\bigoplus(\bigoplus_{\mu\in P^l_{+},\mu\neq n\Lambda_{0}}M^{(n,l)}(\bar{\mu})\otimes W^{(n,l)}(\bar{\mu})),
\end{equation}
where  $P(n\Lambda_0)$ is the set of weights of $L_{\widehat{\sl_{l}}}(n,0)$,
and $$P^l_{+}=\{\la\in P(n\Lambda_{0})\;|\; \la+\delta\notin P(n\Lambda_{0}), \ (\la,\be_{i})\in\Z_{\geq 0}, i=0,1,\cdots,l-1\},$$
 $\be_{0},\cdots,\be_{l-1}$ are the simple roots of the affine Lie algebra $\widehat{\sl_{l}}$.

 Recall that $\widetilde{W}(0)\cong K(\sl_{l},n)$.
Regarding $V_{\widetilde{N}_{l}^n}$ as a $ \widetilde{W}(0)$-module, by (\ref{de4}) we have
$$
V_{{\tilde{N}}_l^{n}}(0)=\{v\in V_{\widetilde{N}_{l}^n}| \widetilde{\omega}''_{1}v=0\}\supseteq C_{V_{\widetilde{A}_{l-1}^{\times n}}}(L_{\widehat{\sl_{l}}}(n,0)),
$$
where $\widetilde{\omega}''$ is the Virasoro vector of $\widetilde{W}(0)$ inside $V_{\widetilde{N}_{l}^n}$.
 By Lemma \ref{de-pa2}, $V_{\widetilde{N}_{l}^n}\cong V_{N_{n}^l}$ through the isomorphism $\tau$ and $\tau(W(0))=\widetilde{W}(0)$. By (\ref{de5}), we immediately have
 $$
 C_{V_{\widetilde{A}_{l-1}^{\times n}}}(L_{\widehat{\sl_{l}}}(n,0))=K(\sl_{n},l).$$
Similarly,
$$
 C_{V_{{A}_{n-1}^{\times l}}}(L_{\widehat{\sl_{n}}}(l,0))=C_{V}(U)=K(\sl_{l},n).$$ 
 \end{proof}

 \section{Level-Rank Duality for General Case}
\setcounter{equation}{0}

In this section, we will establish the level-rank duality for general case.
 As in Section 3, let $V$ be the tensor product vertex operator algebra
$$V=L_{\widehat{\frak g}}(l_{1},0)\otimes L_{\widehat{\frak g}}(l_{2},0)\otimes \cdots\otimes L_{\widehat{\frak g}}(l_{m},0)$$
and
$U\cong L_{\widehat{\sl_n}}(l,0)$  the vertex operator subalgebra of $V$ generated by $\sl_{n}(\C)$ which is diagonally imbedded into $V_{1}$, where $\frak{g}=\sl_{n}(\C)$, $l_1,\cdots,l_m\in\Z_{+}$, $l=l_1+\cdots+l_m$, $m\in\Z_{\geq 2}$.

 Let $A_{n-1}^{\times l}$, $\widetilde{A}_{l-1}^{\times n}$, $N_{n}^{l}$, and $\widetilde{N}_l^{n}$ be the same as in Section 3. Denote
 $$
 s_0=0, \ s_j=l_1+l_2+\cdots+l_j, \ 1\leq j\leq m.
 $$

 For  $r=1,2,\cdots,m$, define
 $$
 A_{n}^{\times l,r}=\bigoplus_{i=1}^{n-1}\bigoplus_{j=s_{r-1}+1}^{s_{r}}\Z\al^{ij},$$
 
 For $ r=1,2,\cdots,m$ such that $s_r\geq 2$, define
 $$ \  \widetilde{A}_{l,r}^{\times n}=\bigoplus_{i=s_{r-1}+1}^{s_r-1}\bigoplus_{j=1}^{n}\Z\be^{ij},
 $$
 $$
 {N}_n^{l,r}=\bigoplus_{i=1}^{n-1}\bigoplus_{j=s_{r-1}+1}^{s_r-1}\Z(\al^{ij}-\al^{i,j+1}), \ \
 \widetilde{N}_{l,r}^{n}=\bigoplus_{i=s_{r-1}+1}^{s_r-1}\bigoplus_{j=1}^{n-1}\Z(\be^{ij}-\be^{i,j+1}),
 $$
 $$
 {K}_{r}=\bigoplus_{i=s_{r-1}+1}^{s_r-1}\Z(\be^{i1}+\cdots+\be^{in}).
 $$

 Recall that the simple vertex operator algebra $L_{\widehat{\sl_n}}(l,0)$ is naturally imbedded into the lattice vertex operator algebra $V_{A_{n-1}^{\times l}}$ and $L_{\widehat{\sl_l}}(n,0)$ is naturally imbedded into $V_{A_{l-1}^{\times n}}$.  We now imbed $L_{\widehat{\sl_n}}(l_k,0)$ into the vertex operator algebra $V_{A_{n}^{\times l,k}}$.  For $l_k\geq 2$, we  also imbed $L_{\widehat{\sl_{l_k}}}(n,0)$ into $V_{\widetilde{A}_{l,k}^{\times n}}$.

 Denote
 $$\frak{h}_l=\bigoplus_{i=1}^{l-1}\C(\be^{i1}+\cdots+\be^{in}).
 $$
 Then $\frak{h}_l$ is  a Cartan  subalgebra of $L_{\widehat{\sl_{l}}}(n,0)$. Set
 $$
 \frak{h}_{\underline{\ell}}=\{\be\in\frak{h}_l \; | \; (\be, \be^{i1}+\cdots+\be^{in})=0, \ {\rm for} \  s_{k}\geq 2,  s_{k-1}+1\leq i\leq s_k-1, \ 1\leq k\leq m\}
 $$
 and
 $$
 \frak{l}_{\underline{\ell}}=\left(\bigoplus_{k=1,l_k\geq 2}^{m}L_{\widehat{\sl_{l_k}}}(n,0)_{1}\right)\bigoplus\frak{h}_{\underline{\ell}}.
 $$
Then $\frak{l}_{\underline{\ell}}$ is a Levi subalgebra of $\sl_l(\C)$ and $\frak{h}_{\underline{\ell}}$ is the center of $\frak{l}_{\underline{\ell}}$ which is contained in the (fixed) Cartan subalgebra of $\sl_{l}$. Denote by $L_{\widehat{\frak{l}_{\underline{\ell}}}}(n,0)$ the vertex operator subalgebra of $L_{\widehat{\sl_l}}(n,0)$ generated by $\frak{l}_{\underline{\ell}}$.
It is easy to see that
\begin{equation} \label{levi-tensor}
L_{\widehat{\frak{l}_{\underline{\ell}}}}(n,0) \cong \left(\bigotimes_{k=1,l_k\geq 2}^{m}L_{\widehat{\sl_{l_k}}}(n,0)\right)\bigotimes L_{\widehat{\frak{h}}_{\underline{\ell}}}(n,0),
\end{equation}
 where $L_{\widehat{\frak{h}}_{\underline{\ell}}}(n,0)$ is the Heisenberg vertex operator subalgebra of $L_{\widehat{\sl_{l}}}(n,0)$ generated by $\frak{h}_{\underline{\ell}}$. We denote
$$
K(\sl_l, \frak{l}_{\underline{\ell}},n)=C_{L_{\widehat{\sl_{l}}}(n,0)}(L_{\widehat{\frak{l}_{\underline{\ell}}}}(n,0)).
$$
\begin{defn}
$K(\sl_l, \frak{l}_{\underline{\ell}},n)$ is called a   relative parafermion vertex operator algebra.
\end{defn}

We have the following level-rank duality  which characterizes  relations between tensor decomposition and relative parafermion vertex operator algebras.
\begin{theorem}\label{lem4.2} We have
$$
C_{L_{\widehat{\sl_{n}}}(l_1,0)\otimes\cdots\otimes L_{\widehat{\sl_{n}}}(l_m,0)}(L_{\widehat{\sl_{n}}}(l,0))\cong K(\sl_l, \frak{l}_{\underline{\ell}},n).
$$
\end{theorem}
\begin{proof} We just prove the theorem for $m=2$, since for  $m>2$ the argument is quite similar. Now assume that $m=2$. Then $l=l_1+l_2$. We may assume that $l_1,l_2\geq 2$. Recall the natural imbedding
$$
L_{\widehat{\sl_{n}}}(l,0)\subseteq L_{\widehat{\sl_{n}}}(l_1,0)\otimes L_{\widehat{\sl_{n}}}(l_2,0)\subseteq L_{\widehat{\sl_{n}}}(1,0)^{\otimes l_1}\otimes  L_{\widehat{\sl_{n}}}(1,0)^{\otimes l_2}\cong L_{\widehat{\sl_{n}}}(1,0)^{\otimes l}.
$$
Let $P(l_i\Lambda_{0})$ be the set of weights of the  $\widehat{\sl_n}$-module $L_{\widehat{\sl_{n}}}(l_i,0)$, $i=1,2$. Denote
$$
P^{n,l_i}_{+}=\{\la\in P(l_i\Lambda_{0})\; |\; \la {\rm \ is \ integral \ dominate \ and \ }  \la+\delta\notin P(l_i\Lambda_{0})\}, \ i=1,2.
$$
 By Theorem \ref{lr-dual2}, we have
$$
\begin{array}{ll}
 &L_{\widehat{\sl_{n}}}(1,0)^{\otimes l_1}\otimes L_{\widehat{\sl_{n}}}(1,0)^{\otimes l_2}\\
&\quad =  [\bigoplus_{\la\in P^{n,l_1}_{+}}(L_{\widehat{\sl_{n}}}(l_1,\bar{\lambda})\otimes M^{(l_1,n)}(\bar{\lambda}))]\otimes
     [\bigoplus_{\mu\in P^{n,l_2}_{+}}(L_{\widehat{\sl_{n}}}(l_2,\bar{\mu})\otimes M^{(l_2,n)}(\bar{\mu}))]\\
     &\quad = \bigoplus_{\la\in P^{n,l_1}_{+}}\bigoplus_{\mu\in P^{n,l_2}_{+}}[L_{\widehat{\sl_{n}}}(l_1,\bar{\lambda})\otimes L_{\widehat{\sl_{n}}}(l_2,\bar{\mu})\otimes  M^{(l_1,n)}(\bar{\lambda})\otimes M^{(l_2,n)}(\bar{\mu})]\\
    &\quad  =  \bigoplus_{\la\in P^{n,l_1}_{+}}\bigoplus_{\mu\in P^{n,l_2}_{+}}[\bigoplus_{\gamma\in P_{+}(\la,\mu)}(L_{\widehat{\sl_{n}}}(l,\bar{\gamma})
     \otimes M(\la,\mu,\bar{\gamma}))]\otimes  M^{(l_1,n)}(\bar{\lambda})\otimes M^{(l_2,n)}(\bar{\mu}),
\end{array}
$$
where $P_{+}(\la,\mu)$ is a set of integral dominant weights  of the $\widehat{\sl_{n}}$-module $L_{\widehat{\sl}_{n}}(l_1, \lambda)\otimes L_{\widehat{\sl}_{n}}(l_2, \mu)$. $M(0,0,0)$ is the commutant vertex operator algebra $C_{V}(U)$ and $M(\la,\mu,\bar{\gamma})$ are $C_{V}(U)$-modules, $\la\in P^{n,l_1}_{+}, \mu\in P^{n,l_2}_{+}, \gamma\in P_{+}(\la,\mu)$. So
$$
\begin{array}{c}
 C_{L_{\widehat{\sl_{n}}}(1,0)^{\otimes l}}(L_{\widehat{\sl_n}}(l,0))
= \bigoplus_{\la\in P^{n,l_1}_{+}}\bigoplus_{\mu\in P^{n,l_2}_{+}}
M^{(l_1,n)}(\bar{\lambda})\otimes M^{(l_2,n)}(\bar{\mu})\otimes M(\la,\mu,0).
\end{array}
$$
Note that $$M^{(l_1,n)}(0)=C_{L_{\widehat{\sl_{n}}}(1,0)^{\otimes l_1}}(L_{\widehat{\sl_{n}}}(l_1,0)), \
 M^{(l_2,n)}(0)=C_{L_{\widehat{\sl_{n}}}(1,0)^{\otimes l_2}}(L_{\widehat{\sl_{n}}}(l_2,0)).$$

On the other hand, we have
\begin{equation} \label{levi-decomp1}
\begin{array}{rl}
 L_{\widehat{\sl_l}}(n,0)
&= \bigoplus_{\la^{1},\la^{2},\la^{3}}L_{\widehat{\sl_{l_1}}}(n,\overline{\la^{1}})\otimes L_{\widehat{\sl_{l_2}}}(n,\overline{\la^{2}})\otimes L_{\widehat{\frak{h}}_{\underline{\ell}}}(n,\la^{3})\otimes \widetilde{M}(\la^{1},\la^{2},\la^{3})\\
&=   \bigoplus_{\la^1,\lambda^2,\la^3}[\bigoplus_{\gamma^1\in K_{1}^{\circ}/K_{1}}V_{K_1+\overline{\la^{1}}+\gamma^1}\otimes U^{(n,l_1)}(\overline{\la^{1}},\gamma^1)]
\otimes \\
&\quad [\bigoplus_{\gamma^2\in K_{2}^{\circ}/K_{2}}V_{K_2+\overline{\la^{2}}+\gamma^2}\otimes U^{(n,l_2)}(\bar{\la^{2}},\gamma^2)]
     \otimes L_{\widehat{\frak{h}}_{\underline{\ell}}}(n,\la^{3})\otimes \widetilde{M}(\la^{1},\la^{2},\la^{3}),
\end{array}
\end{equation}
where $L_{\widehat{\sl_{l_i}}}(n,\overline{\la^{i}})$ are irreducible 
$L_{\widehat{\sl_{l_i}}}(n,0)$-modules for $i=1,2$ and 
$L_{\widehat{\frak{h}}_{\underline{\ell}}}(n,\la^{3})$ is an irreducible 
$L_{\widehat{\frak{h}}_{\underline{\ell}}}(n,0)$-module with lowest weight 
$\la^{3}\in \mathbb{Q}_{+}$. $M(\la^{1},\la^{2}, \la^{3})$ are 
modules of $C_{L_{\widehat{\sl_l}}(n,0)}(L_{\widehat{\frak{l}_{\underline{\ell}}}}(n,0))$ 
and $\widetilde{M}(0,0,0)=C_{L_{\widehat{\sl_l}}(n,0)}(L_{\widehat{\frak{l}_{\underline{\ell}}}}(n,0))=K(\sl_l, \frak{l}_{\underline{\ell}},n)$. 
So we have
\begin{equation} K(\sl_l,n)
=    \bigoplus_{\la^1,\lambda^2}U^{(n,l_1)}(\overline{\la^{1}},-\overline{\la^1})
\otimes U^{(n,l_2)}(\overline{\la^{2}},-\overline{\la^2})\otimes \widetilde{M}(\la^{1},\la^{2},0).
\end{equation}
   By Theorem \ref{lr-dual2}, we have
$$ K(\sl_l,n)\cong C_{L_{\widehat{\sl_{n}}}(1,0)^{\otimes l}}(L_{\widehat{\sl_n}}(l,0)).
$$
Let $\tau$ be the isomorphism from $C_{L_{\widehat{\sl_{n}}}(1,0)^{\otimes l}}(L_{\widehat{\sl_n}}(l,0))$ to $K(\sl_l,n)$. Then
$$
\begin{array}{l}
 \tau(\bigoplus_{\la\in P^{n,l_1}_{+}}\bigoplus_{\mu\in P^{n,l_2}_{+}}M^{(l_1,n)}(\bar{\lambda})\otimes M^{(l_2,n)}(\bar{\mu})\otimes  M(\la,\mu,0))\\
     \quad \quad=   \bigoplus_{\la^1,\lambda^2}U^{(n,l_1)}(\overline{\la^{1}},-\overline{\la^1})
\otimes U^{(n,l_2)}(\overline{\la^{2}},-\overline{\la^2})\otimes \widetilde{M}(\la^{1},\la^{2},0)
\end{array}
$$
and
$$
\tau(M^{(l_1,n)}(0))=U^{(n,l_1)}(0,0)=K(\sl_{l_1},n), $$$$ \tau(M^{(l_2,n)}(0))=U^{(n,l_2)}(0,0)=K(\sl_{l_2},n).
$$
 Considering $\bigoplus_{\la\in P^{n,l_1}_{+}}\bigoplus_{\mu\in P^{n,l_2}_{+}} M^{(l_1,n)}(\bar{\lambda})\otimes M^{(l_2,n)}(\bar{\mu})\otimes M(\la,\mu,0)$ and $\bigoplus_{\la^1,\lambda^2}U^{(n,l_1)}(\overline{\la^{1}},-\overline{\la^1})
\otimes U^{(n,l_2)}(\overline{\la^{2}},-\overline{\la^2})\otimes \widetilde{M}(\la^{1},\la^{2},0)$ as $K(\sl_{l_1},n)\otimes K(\sl_{l_2},n)$-modules and using (2) of Lemma \ref{decomp2}, we have
$$
\tau(M(0,0,0))= \widetilde{M}(0,0,0).
$$
That is,
$$
C_{L_{\widehat{\sl_{n}}}(l_1,0)\otimes L_{\widehat{\sl_{n}}}(l_2,0)}(L_{\widehat{\sl_{n}}}(l,0))\cong C_{L_{\widehat{\sl_{l}}}(n,0)}(L_{\widehat{\sl_{l_1}}}(n,0)\otimes L_{\widehat{\sl_{l_2}}}(n,0)\otimes L_{\widehat{\frak{h}}_{\underline{\ell}}}(n,0)).
$$
\end{proof}

\begin{cor} For any $ \underline{\ell}$ and $n$, $C_{L_{\widehat{\frak{sl}_{|\underline{\ell}|}}}(n, 0)}(K(\frak{sl}_{|\underline{\ell}|}, \frak{l}_{\underline{\ell}}, n)) 
$  is isomorphic to the tensor product of the vertex operator algebras corresponding the  semi-simple Lie algebras and a lattice vertex operator algebra, and in particular is a rational vertex operator algebra.
\end{cor}
\begin{proof} One notes that $C_{L_{\widehat{\frak{sl}_{|\underline{\ell}|}}}(n, 0)}(K(\frak{sl}_{|\underline{\ell}|}, \frak{l}_{\underline{\ell}}, n)) \supseteq L_{\widehat{\frak{l}_{\underline{\ell}}}}(n,0)$.  Then by using \eqref{levi-tensor} and  \eqref{levi-decomp1} by setting $ (\lambda^1, \lambda^2, \lambda^3)=(0, 0, 0)$,  a similar argument to the proof of \eqref{de11} and \eqref{decomp1}, with $n$ and $l$ switched, will prove the corollary. The vertex operator algebra corresponding to the semi-simple Lie algebra is $\left(\bigotimes_{k=1,l_k\geq 2}^{m}L_{\widehat{\sl_{l_k}}}(n,0)\right)$ which is rational. The lattice vertex operator algebra corresponds to the Heisenberg vertex operator algebra $ L_{\widehat{\frak{h}}_{\underline{\ell}}}(n,0)$.
\end{proof}
\section{Duality Pairs and Reciprocity}
\subsection{} Let $(V, Y, \mbf{1}, \omega)$ be a vertex operator algebra. A semi-conformal vertex subalgebra of $V$ is a vertex subalgebra $U$ together with a conformal vector $\omega'\in U$ such that $ (U, Y|_U, \mbf{1}, \omega')$ is a vertex operator algebra and $ L(n)|_U=L'(n)|_U$ for all $ n\geq -1$, where $ L'(n)\in \End(V) $ is defined by $ Y(\omega', z)=\sum L'(n)z^{-n-2}$. Note that if $ L(-2)|_U=L'(-2)|_U$, then $\omega=\omega'$. In this case, $V$ is called a conformal extension of $U$ and $U$ is called a conformal vertex operator subalgebra. One can check the proof of \cite[Theorem 3.11.12]{LL} that if $U$ is a semi-conformal subalgebra of $V$, then $ C_V(U)$ is also a semi-conformal vertex subalgebra of $V$ with conformal vector $\omega-\omega'$. A semi-conformal subalgebra $(U, \omega')$ of $V$ is called closed if $ C_V(C_V(U))=U$. We will denote $ \bar{U}=C_V(C_V(U))$ if $V$ is understood from the context. If $U$ is closed in $V$, then $ U=\ker(L(-1)-L'(-1): V\rightarrow V)$ is uniquely determined by $ \omega'$. In fact the converse also holds.

\begin{lem} If $(U, \omega')$ and $(U, \omega'')$ are two semi-conformal subalgebras  of $V$, then $ \omega'=\omega''$.
\end{lem}
\begin{proof} By definition we have $ L'(n)|_U=L(n)|_U=L''(n)|_U$ for all
$ n\geq -1$. In particular, $(U, \omega')$ is  a semi-conformal subalgebra
of $(U, \omega'')$ and the vice verse.
By \cite[Corollary 3.11.11]{LL}, $C_{(U, \omega')}(U, \omega'')=\C{\bf 1}$.
By \cite[Theorem 3.11.12]{LL},
$\omega'-\omega''\in C_{(U, \omega')}(U, \omega'')$. So $\omega'=\omega''$.
\end{proof}

\subsection{}  Let $U$ be a vertex operator algebra and $ M$  a $V$-module. We denote by $ \Irr_M(U)$  the set of isomorphism classes of irreducible $U$-modules $N$ such that $ \Hom_{U}(N,M)\neq 0$. If $M$ is a semi-simple $U$-module, then we have
\[ M\cong \bigoplus_{N\in \Irr_M(U)} N\otimes \Hom_{U}(N,M)\]
as $U$-modules.  If $ U$ is a semi-conformal subalgebra of a vertex operator algebra $V$ and $ M$ is a $V$-module, then $ \Hom_{U}(N, M)$ is a $C_V(U)$-module and is denoted as $\rho_M(N)$ if $ U$ and $ V$ are understood from the context.
Let $V$ be a simple vertex operator algebra, a pair of semi-conformal simple vertex operator subalgebras $(U^1, U^2) $ of $V$ is called a {\em duality pair  in $V$ with respect to a $V$-module $M$}  if (a) $U^i=C_V(U^j)$ with $i\neq j$ and (b)
\[ M=\bigoplus N^1\otimes \rho_M(N^1)\]
as a $U^1\otimes U^2$-module. Here the summation is taken over all  $N^1\in\Irr_M(U^1)$ of irreducible $U_1$-submodules such that $ \Hom_{U^1}(N^1, M)\neq 0$  and $\rho_M: \Irr_M(U^1)\rightarrow \Irr_M(U^2)$ is a bijection.

In case $M=V$, we simply say that  $ (U^1, U^2)$ is a {\em duality pair in $V$}. This concept is motivated by the duality pairs of reductive groups defined by Howe in \cite{howe:reciprocity}

 Note that since both $U^1$ and $U^2$ are simple,  $ U^1\otimes U^2$ is a (conformal) vertex operator subalgebra of $V$. Also for any simple vertex operator algebra $V$, if one of $U^i$ is a lattice vertex operator algebra, then the rationality and the fusion rules of lattice vertex operators  state in Theorem~\ref{Dong-L} imply that (b) is a consequence of (a) for all irreducible $V$-modules $M$.
 We now summarize the results from Section 3 and Section 4:

 Lemma 3.2  and formula \eqref{decomp1} implies that $ C_V(C_V(V_{\f{h}_n}(l,0)))=
 V_{K}$. Hence $ (V_{K}, V_{N^l_{n}})$ is a duality pair in $V=L_{\wg}(1,0)^{\otimes l}=V_{A_{n-1}^{\times l}}$ with respective the irreducible module $V_{A_{n-1}^{\times l}}$. This is the case that when an even lattice $L=L'\perp L''$ then $ (V_{L'}, V_{L''})$ is a duality pair  in $ V_{L}$ with respect to $V_L$.

 Theorem \ref{lr-dual2} implies that $(K(\f{sl}_n, l), K(\f{sl}_l, n))$ is a duality pair in $V_{N^l_n}$ with respect to $V_{N^l_n}$.

\begin{theorem}\label{section5-t1} $(K(\f{sl}_l, n), V_{\sqrt{n}A_{l-1}})$ is a duality pair in $ L_{\widehat{\f{sl}_l}}(n,0)$.
\end{theorem}

By Theorem \ref{lr-dual2} and Theorem \ref{section5-t1} we have
\begin{cor} $(L_{\widehat{\f{sl}_n}}(l, 0), K(\f{sl}_l, n))$ is a duality pair in $L_{\widehat{\f{sl}_n}}(1, 0)^{\otimes l}=V_{A_{n-1}^{\times l}}$.
\end{cor}

We have the following chain of semi-conformal vertex subalgebras:
\[ V_{\sqrt{l}A_{n-1}}\subseteq L_{\widehat{\f{sl}_n}}(l,0)\subseteq L_{\widehat{\f{sl}_n}}(l_1,0)\otimes\cdots \otimes L_{\widehat{\f{sl}_n}}(l_s,0) \subseteq L_{\widehat{\f{sl}_n}}(1,0)^{\otimes l}=V_{A_{n-1}^{\times l}},\]
where $l=l_1+\cdots+l_{s}$. Note that each of the vertex operator algebras is rational. Although the following two vertex operator algebras
\[ V_{\widehat{\f{h}}}(l, 0) \subseteq L_{\widehat{\f{l}_{\underline{\ell}}}}(l,0) \] in $L_{\widehat{\f{sl}_n}}(l,0) $
are not rational,  the maximal conformal extensions
 $\bar{V}_{\widehat{\f{h}}}(l, 0) \subseteq \bar{L}_{\widehat{\f{l}_{\underline{\ell}}}}(l,0)$ in $L_{\widehat{\f{sl}_n}}(l,0)$ are rational. The following discussions of  duality pairs will use this fact.

Let $ W=L_{\widehat{\f{sl}_n}}(1,0)^{\otimes l}$ and $ V=L_{\widehat{\f{sl}_n}}(l_1,0)\otimes\cdots \otimes L_{\widehat{\f{sl}_n}}(l_s,0)$ and $ U=L_{\widehat{\f{sl}_n}}(l,0)$.  We will study   the following  pairs in a sequel paper:
\[ (V, C_W(V))\; \text{in} \; W, \quad  (U, C_V(U)) \; \text{in}\; V, \quad (U, C_W(U)) \  {\rm in} \ W.\]

\subsection{} Let $(U^1,U^2)$ and $(U'^1, U'^2)$ be two  duality pairs in $V$ with respect to an irreducible $V$-module $M$. We use $ \rho_M :\Irr_M(U^1)\rightarrow \Irr_M(U^2) $ and  $ \rho'_M :\Irr_M(U'^1)\rightarrow \Irr_M(U'^2) $ respectively to denote the correspondence. In case $ U^1\subseteq U'^1$, then these two duality pairs will have certain {\em see and saw} duality property. For reductive groups, various constructions of pairs of duality pairs are discussed in \cite{Ku} and vertex operator algebra versions of these constructions are expected to carry over as well.  
\begin{theorem}  If $ U^1\subseteq U'^1$, then $(U^1, U^2\cap U'^1)$ is a duality pair in $U'^1$ with respect to any $ M'^1\in \Irr_{M}(U'^1)$ and $ (U^2\cap U'^1, U'^2)$ is a duality pair in $U^2$ with respect to any $M^2\in \Irr_{M}(U^2)$.
\end{theorem}
\begin{proof} By (a) of the duality pair condition, $U^1\subseteq U'^1$ implies that
$U^2\supseteq U'^2$. Thus $ U'^1\cap U^2=U'^1\cap C_{V}(U^1)=C_{U'^1}(U^1)$.
Similarly, $ U^1=C_V(U^2)=U'^1\cap C_V(U^2)=C_{U'^1}(U'^1\cap U^2)$. Thus the
condition (a) holds for the pair $ (U^1, U'^1\cap U^2)$. To check the condition (b),
we first have, for each $ M^1 \in \Irr_M(U^1)$,
\[\Hom_{U^1}(M^1, M)=\oplus _{M'^1\in \Irr_{M}(U'^1)} \Hom_{U^1}(M^1, M'^1)\otimes \rho'_{M}(M'^1) \]
as $ U^2$-modules.
Thus we get the following decomposition of $M$ as $U^1$-modules,
\[ M=\bigoplus_{M^1\in \Irr_M(U^1)}\bigoplus_{M'^1\in \Irr_M(U'^1)}  M^1\otimes \Hom_{U^1}(M^1, M'^1)\otimes \rho'_{M}(M'^1).\]
Now comparing with the decomposition of $M$ as $ U'^1$-modules
\[ M=\bigoplus_{M'^1\in \Irr_M(U'^1)}  M'^1 \otimes \rho'_{M}(M'^1)\]
and by applying the functor $ M'^1=\Hom_{U'^2}(\rho'_M(M'^1), M)$ to get
\[ M'^1=\bigoplus_{M^1\in \Irr_M(U^1)} M^1\otimes \Hom_{U^1}(M^1, M'^1)\]
which is an isomorphism of $ U'^1$-modules. Note that $M^1 \in \Irr_{M'^1}(U^1)$ if and only if $ \Hom_{U^1}(M^1, M'^1)\neq 0$. Thus we have the bijection $ \Irr_{M'^1}(U^1)\rightarrow \Irr_{M'^1}(U^2\cap U'^1)$. The other duality pair
follows from a similar argument by replacing the pairs $U^1\subseteq U'^1$ with $U'^2\subseteq U^2$.
\end{proof}

In the above proof, we notice that $U'^2\subseteq  U^2$. Hence any  $U^2$-module $ \rho_M(M^1)$ is a $U'^2$-module. Thus we have
\begin{cor} With the assumption of the theorem,   for any $ M^1 \in \Irr_M(U^1)$ and $M'^1\in \Irr_M(U'^1)$, we have an isomorphism of  $ U^2\cap U'^1$-modules
\[\Hom_{U^1}(N, N')\cong \Hom_{U'^2}(\rho'_{M}(N'), \rho_M(N)).\]
\end{cor}
This isomorphism provides a reciprocity law for vertex operator algebras in the sense of \cite{howe:reciprocity} which was for reductive groups.

\subsection{} We now apply reciprocity law in the above corollary to various cases we discussed in this paper.

(1) In $ L_{\widehat{\f{sl}_l}}(n,0)$, we consider the two duality pairs $ (U'^1, U'^2)=(\bar{L}_{\widehat{\f{l}_{\underline {\ell}}}}(n,0), K(\f{sl}_l, \f{l}_{\underline{\ell}}, n))$  and 
 $(U^1,U^2)=(V_{\sqrt{n}A_{l-1}}, K(\f{sl}_l, n))$. Then  we have isomorphisms
\[ C_{\bar{L}_{\widehat{\f{l}_{\underline {\ell}}}}(n,0)}(V_{\sqrt{n}A_{l-1}})
\cong C_{K(\f{sl}_l, n)}(K(\f{sl}_l, {\f{l}_{\underline{\ell}}}, n))\] 
which is isomorphic further to $K(\f{sl}_l, n)\cap \bar{L}_{\widehat{\f{l}_{\underline {\ell}}}}(n,0) $ for any $\underline{\ell}=(l_1,\cdots, l_s)$.

(2) Similarly if we apply the corollary to the two duality pairs  $(U^1, U^2)=(V_{\sqrt{l} A_{n-1}}, V_{N^l_n}) $ and $(U'^1, U'^2)=(L_{\widehat{\f{sl}_n}}(l,0), K(\f{sl}_{l}, n))$ in $ V_{A_{n-1}^{\times l}}$ with respective to $ M= V_{A_{n-1}^{\times l}}$. We get
\[  K(\f{sl}_n, l)= C_{L_{\widehat{\f{sl}_n}}(l,0)}(V_{\sqrt{l}A_{n-1}})\cong C_{V_{N^l_n}}(K(\f{sl}_l, n))\]
which  follows from the proof of Theorem 3.10.

(3) Now let us set $(U^1, U^2)=(L_{\widehat{\f{sl}_n}}(l,0), K(\f{sl}_l, n))$  and $(U'^1, U'^2)=(L_{\widehat{\f{sl}_n}}(l_1,0)\otimes \cdots \otimes L_{\widehat{\f{sl}_n}}(l_s,0), K(\f{sl}_{l_1}, n)\otimes \cdots\otimes K(\f{sl}_{l_s}, n))$ in $ V_{A_{n-1}^{\times l}}$ with respective to $ M= V_{A_{n-1}^{\times l}}$. Here we use the convention that $ K(\f{sl}_1, n)=\C\mbf{1}$. Then by Theorem \ref{lem4.2} we get
\[K(\f{sl}_l, \f{l}_{\underline{\ell}}, n)\cong C_{L_{\widehat{\f{sl}_n}}(l_1,0)\otimes \cdots \otimes L_{\widehat{\f{sl}_n}}(l_s,0)}( L_{\widehat{\f{sl}}_n}(l,0))\cong C_{K(\f{sl}_l, n)}(K(\f{sl}_{l_1}, n)\otimes \cdots\otimes K(\f{sl}_{l_s}, n)). \]

\section{Applications and Examples}
In this section we summarize known results on the question of rationality of the coset construction and parafermions using  Theorem~\ref{thm:1.1}.

The structure of parafermion vertex operator algebras has been studied extensively recently ( see \cite{ALY, DLWY, DLY2,DW1, DW2} etc.). Most of the results concentrated on $K(\frak{sl}_2, l)$. The next result gives rationality of parafermion vertex operator algebras  for all $\frak{sl}_n$ in the level 2 case.

\begin{theorem} For any $n\geq 2$, the parafermion $ K(\frak{sl}_n, 2)$ is always rational and it has $2^{n-2}(n+1)$ irreducible representations.
\end{theorem}
\begin{proof} By Theorem~\ref{thm:1.1}, $K(\frak{sl}_n, 2)\cong C_{L_{\widehat{\frak{sl}_2}}(1,0)^{\otimes n}}(L_{\widehat{\frak{sl}_2}}(n,0))$.  Now the theorem follows from \cite[Thm. 4.16, Thm. 5.5]{JL}.
\end{proof}

Conversely the known results of parafermions  can be transferred to coset constructions.
\begin{theorem} (1) $C_{L_{\widehat{\frak{sl}_n}}(1,0)^{\otimes l}}(L_{\widehat{\frak{sl}_n}}(l,0))$ is always $C_2$-cofinite.

	(2) The Zhu algebra $A(C_{L_{\widehat{\frak{sl}_n}}(1,0)^{\otimes 2}}(L_{\widehat{\frak{sl}_n}}(2,0)) )$ is semisimple, commutative, of dimension $n(n+1)/2$. In particular,  $ C_{L_{\widehat{\frak{sl}_n}}(1,0)^{\otimes 2}}(L_{\widehat{\frak{sl}_n}}(2,0))$ has exactly $n(n+1)/2$ non-isomorphic irreducible representations.
	
	(3) $C_{L_{\widehat{\frak{sl}_n}}(1,0)^{\otimes 2}}(L_{\widehat{\frak{sl}_n}}(2,0)) $ is rational for $n\leq 6$.

\end{theorem}

\begin{proof} (1) follows from \cite[Thm. 10.5]{ALY}.  (2) follows from \cite[Thm. 8.2]{ALY} and \cite{DLY2} and (3) follows from \cite{DLY2}.
\end{proof}

\end{document}